\pdfoutput=1
\documentclass{amsart}
\usepackage{amsmath,
 amssymb,
 fullpage,
 mathtools,
 mathpazo,
 thmtools,
 enumerate
}
\usepackage[
 pagebackref,
 colorlinks,
 allcolors=blue
]{hyperref}

\usepackage[capitalize]{cleveref}

\newlength{\bibitemsep}
\newlength{\bibparskip}
\let\oldthebibliography\thebibliography
\renewcommand\thebibliography[1]{%
 \oldthebibliography{#1}%
 \setlength{\parskip}{\bibparskip}%
 \setlength{\itemsep}{\bibitemsep}%
}

\DeclareMathOperator{\ord}{ord}

\newcommand{\QQ}{\mathbb{Q}}

\newcommand{\ediv}{\mathrel{\|}} 

\newcommand{\zpm}{\mathbb{Z}_p^{\times}}
\newcommand{\zpa}{\mathbb{Z}_p}

\newtheorem{defi}{Definition}[section]
\newtheorem{thm}[defi]{Theorem}
\newtheorem{prop}[defi]{Proposition}
\newtheorem{lemma}[defi]{Lemma}
\newtheorem{cor}[defi]{Corollary}
\newtheorem{remark}[defi]{Remark}
\newtheorem{letthm}{Theorem}

\newenvironment{revised}{}{}
\newcommand{\rev}[1]{#1}

\begin{document}

\title[Rankin--Selberg $L$-functions in universal deformation families]{$p$-adic Rankin--Selberg $L$-functions in universal deformation families and functional equations}
 \renewcommand{\emailaddrname}{\emph{Email}}
 \renewcommand{\urladdrname}{\emph{ORCID}}

\author{Zeping Hao}
\address[Z.H.]{Academy of Mathematics and Systems Science, Chinese Academy of Sciences,
No. 55 Zhongguancun East Road, Beijing, CN 100190}
\email{zepinghao@amss.ac.cn}
 \urladdr{\href{https://orcid.org/0009-0000-6521-5553}{\texttt{orcid.org/0009-0000-6521-5553}}}

\author{David Loeffler}
\address[D.L.]{UniDistance Suisse, Schinerstrasse 18, 3900 Brig, Switzerland}
\email{david.loeffler@unidistance.ch}
 \urladdr{\href{https://orcid.org/0000-0001-9069-1877}{\texttt{orcid.org/0000-0001-9069-1877}}}

\thanks{Z.H is supported by University of Warwick's Chancellor's International Scholarship. D.L is supported by the ERC Consolidator grant “Shimura varieties and the BSD conjecture” (grant ID 101001051).}
\date{}


\maketitle

\begin{abstract}
 We construct a $p$-adic Rankin--Selberg $L$-function associated to the product of two families of modular forms, where the first is an ordinary (Hida) family, and the second an arbitrary universal-deformation family (without any ordinarity condition at $p$). This gives a function on a 4-dimensional base space -- strictly larger than the ordinary eigenvariety, which is 3-dimensional in this case. We prove our $p$-adic $L$-function interpolates all critical values of the Rankin--Selberg $L$-functions for the classical specialisations of our family, and derive a functional equation for our $p$-adic $L$-function by applying a recent deep result of Helm and Moss on universal $\gamma$-factors.
\end{abstract}

\section{Introduction}

 \subsection{Overview}

  To two normalised modular eigenforms $f=\sum\limits_{n \geq 1}a_{n}q^{n}$ and $g=\sum\limits_{n \geq 1} b_{n }q^{n}$ of weights $k > l$, one can attach a Dirichlet $L$-series $L\left( f,g,s\right)=\sum\limits_{n\geq 1}c_{n}n^{-s}$, called the \textit{Rankin--Selberg L-function}, such that $c_{\ell}=a_{\ell}b_{\ell}$ for $\ell$ a prime. A search for its $p$-adic counterparts was initiated in the 80s by Panchishkin \cite{Pan82} and Hida \cite{Hid85}, assuming $f$ is ordinary. Since then it has been a fruitful area of mathematical research. A milestone work is given in Hida's monumental paper \cite{Hid88}, where he constructed a three-variable $p$-adic Rankin--Selberg $L$-function, allowing $f$ and $g$ both to vary in Hida families. This $p$-adic $L$-function is in general \textit{imprimitive}, in the sense that it interpolates the critical values of complex Rankin--Selberg $L$-functions whose local Euler factors at ramified primes do not necessarily agree with the automorphic $L$-factors (as defined in \cite{Jac72} for example).

  Hida's work has been generalised in many directions. For example, one can consider non-ordinary $f$ and $g$, and ask for variations in Coleman families; this is studied in depth in e.g.\ \cite{LZ16, loe18, AI21}. In a different direction, Chen and Hsieh constructed in \cite{CH18} primitive $p$-adic Rankin--Selberg $L$-functions for Hida families, whose local factors agree with those of \cite{Jac72} in all cases, using the results of Fouquet and Ochiai \cite{FO12} on rigidity of automorphic types in Hida families.

  A limitation of the above constructions is that the $p$-adic $L$-functions constructed can only vary in families whose automorphic representations at $p$ are non-super\-cuspidal, with the additional data of a ``$p$-refinement" (corresponding to a 1-dimen\-sional invariant subspace in the local Galois representation at $p$ for Hida families, and in its $\left( \varphi, \Gamma\right)$-module for Coleman families). Consequently, the parameter spaces for these families are the three-dimensional \textit{eigenvarieties}, parametrising pairs of $p$-refined modular forms (with an additional variable for twisting). However, one expects that the existence of a $p$-refinement for the second family should be unnecessary. This is in accordance with general conjectures of Panchishkin, predicting that to define a $p$-adic $L$-function for some Galois representation, it suffices that the Galois representation have a single local subrepresentation (or sub-$(\varphi, \Gamma)$-module) at $p$ of a specific dimension -- a full flag of local subrepresentations is unnecessary. In the Rankin--Selberg case, ordinarity of just one of the two families is sufficient to construct such a subrepresentation.

  Accordingly, the first aim of this paper is to give a construction of $p$-adic Rankin--Selberg $L$-functions that vary over larger parameter spaces, of dimension four, which parametrise pairs $(f, g)$ of modular forms together with an ordinary $p$-stabil\-isation of $f$ (but \emph{not} of $g$). This gives a larger parameter space, as we are imposing a weaker condition; and our $p$-adic $L$-function now covers more points than other $p$-adic Rankin--Selberg $L$-functions in literature so far, as we allow modular forms whose local representations at $p$ are supercuspidal. These parameter spaces are examples of the (mostly conjectural) ``big parabolic eigenvarieties'' introduced by the second author in \cite{loe21}. The existence of such a $p$-adic $L$-function is sketched in \cite{loe21} for families of tame level 1; in the present paper we supply details of the construction and, more importantly, extend the argument to allow general tame levels. For a full statement see Theorem A below.

  The second goal of this paper is to prove a functional equation for our $p$-adic $L$-function. This is more difficult than it might appear, since the $\varepsilon$-factors usually used in formulating functional equations do not seem to vary analytically over universal-deformation families. Hence we use the alternative formulation of functional equations via $\gamma$-factors, rather than $\epsilon$-factors; and we apply a deep result of Helm and Moss \cite{hm17} showing that these $\gamma$-factors at ramified primes (away from $p$) can be interpolated $p$-adically. See Theorem B for the precise statement.

  We hope to consider the ``Selmer-group'' counterpart of these constructions -- defining Selmer groups in these four-parameter families which interpolate Bloch--Kato Selmer groups, and formulating an Iwasawa main conjecture relating these to the $p$-adic $L$-function -- in a future work. We are also optimistic that these results should generalise to Hilbert modular forms.

\subsection{Outline of the construction}

  Let $p > 2$ be a fixed prime, and $S_1, S_2$ finite sets of places $\QQ$ each containing $p$ and $\infty$. Let $\overline{\rho}_{1}: G_{\QQ, S_1} \to \operatorname{GL}_{2}\left( \mathbb{F}\right)$ and $\overline{\rho}_{2}:G_{\QQ, S_2} \to \operatorname{GL}_{2}\left( \mathbb{F}\right)$ be two fixed residual representations over $\mathbb{F}$, which is a finite field of characteristic $p$. We require that these be \textit{\rev{absolutely} irreducible} and \textit{odd}, so they come from modular forms. In addition, we assume $\overline{\rho}_{1}$ to be $p$-ordinary (but we make no such assumption on $\overline{\rho}_{2}$). Mazur's deformation theory \cite{maz89} then gives pairs $\left( \rho^{\ord},\mathcal{R}^{\ord}\right)$ and $\left( \rho^{\operatorname{univ}},\mathcal{R}^{\operatorname{univ}}\right)$, parameterising all ordinary deformations of $\overline{\rho}_{1}$, and all deformations of $\overline{\rho}_{2}$, respectively. Under some mild additional assumptions on $\overline{\rho}_{1}$ and $\overline{\rho}_{2}$ (see below), these universal deformation rings are naturally isomorphic to certain Hecke algebras. By work of \cite{epw05}, the universal ordinary deformation $\rho^{\ord}: G_{\QQ, S_1} \to \operatorname{GL}_{2}\left( \mathcal{R}^{\ord}\right)$ descends to a representation $\rho\left( \mathfrak{a}\right): G_{\QQ, S_1} \to \operatorname{GL}_{2}\left( \mathbb{T}_{\mathfrak{a}}\right)$, where $\mathbb{T}_{\mathfrak{a}}$ is the integral closure of an irreducible component of Hida's universal ordinary Hecke algebra (at a new level).

  We can then attach to $\rho^{\ord}$ and $\rho^{\operatorname{univ}}$ \textit{universal eigenforms} $\mathcal{F}$ and $\mathcal{G}$, which are defined as the product of reciprocal Euler factors of the associated representations at each prime. To avoid the technical difficulty of interpolating Euler factors at bad primes, we exclude those Euler factors from the definition of $\mathcal{G}$, and we write $\mathcal{G}^{[pN]}$ for $\mathcal{G}$ to reflect this depletion ($N$ is some large enough integer related to the tame conductors of $\rho^{\ord}$ and $\rho^{\operatorname{univ}}$). Then the isomorphism between universal deformation rings and Hecke algebras, together with the duality between $p$-adic modular forms and Hecke algebras studied in \cite{gou88}, allow us to identify $\mathcal{F}$ and $\mathcal{G}$ as $p$-adic eigenforms with coefficients in $\mathbb{T}_{\mathfrak{a}}$ and $\mathcal{R}^{\operatorname{univ}}$. In particular, by the construction of $\mathbb{T}_{\mathfrak{a}}$, the $p$-adic eigenform $\mathcal{F}$ can be identified as a primitive Hida family.

  Then we may use Hida's theory to find a linear functional $\lambda_{\mathcal{F}}$ that is dual to $\mathcal{F}$, and hence define a $p$-adic $L$-function $\mathcal{L}$ by $\mathcal{L}\coloneqq \lambda_{\mathcal{F}}\left( e^{\ord}\left( \mathcal{G}^{[pN]} \cdot \boldsymbol{F}^{[p]}\right)\right)$, where $e^{\ord}$ is Hida's ordinary projector, and $\boldsymbol{F}^{[p]}$ is an appropriate $p$-adic family of Eisenstein series defined in \cite[\S5.3]{LLZ14}. (This is a slight over-simplification; actually we will twist $\lambda_{\mathcal{F}}$ and $\mathcal{G}$ by suitable prime-to-$p$ Dirichlet characters, in order to obtain a better-looking interpolation formula.) The first main result of this paper is the following interpolation formula, which shows that our $p$-adic $L$-function does interpolate the automorphic Rankin--Selberg $L$-functions:

\begin{letthm}[Theorem \ref{thm:interpolation formula}]
 There exists a (necessarily unique) meromorphic function $\mathcal{L} \in \operatorname{Frac}\left( \mathbb{T}_{\mathfrak{a}}\right) \hat{\otimes}\mathcal{R}^{\operatorname{univ}}$, such that for all modular points $\left( f, \theta^t \left( g \right) \right) \in \operatorname{Spec}\left( \mathbb{T}_{\mathfrak{a}}\right) \times \operatorname{Spec}\left( \mathcal{R}^{\operatorname{univ}}\right)$ with $f$ having weight $k \geq 2$, $g$ having weight $ l \geq 1$, and $t$ an integer satisfying $0 \leq t \leq k-l-1$, we have
 \begin{align*}
  &\mathcal{L}\left( f,\theta^{t}\left( g\right) \right) \\
  =&     i^{k-l-2t}2^{1-k}N^{2+2t-k+l}\psi^{-1}_{p}\left( N\right) \epsilon_{p}\left( -N \right)    \Lambda^{[pN]}(f, g^{*}, l+t) \\
  & \cdot  \lambda_{p^{b}}( g)\left(\frac{p^{t+1}}{\alpha} \right)^b\frac{P_{p}\left( g, p^{t}  \alpha^{-1}\right)}{P_{p}\left(g^{*}, \alpha p^{-l-t}\right)\mathcal{E}^{\operatorname{ad}}\left( f_{\alpha}\right)\langle f, f \rangle_{N_{1}p^{a}}},
 \end{align*}
 where
 \begin{itemize}
  \item  $a$ (resp. $b$) is the power of $p$ dividing the level at which $f$ (resp. $g$) is new, and $\lambda_{p^{b}}\left( g\right)$ is the Atkin--Lehner pseudo-eigenvalue of $g$ at $p^{b}$ ,
  \item $\psi_{p}$ (resp. $\epsilon_{p}$) is the $p$-part character of $f$ (resp. $g$),
  \item $\alpha$ is the unique ordinary root of the Hecke polynomial of $f$ at $p$,
  \item $P_{p}\left( g,X\right)$ is the polynomial satisfying $P_{p}\left(g,X \right)^{-1}=\sum\limits_{u \geq 0} X^{u}a_{p^{u}}\left( g\right)$,
  \item $\mathcal{E}^{\operatorname{ad}}\left( f_{\alpha}\right)$ (the adjoint Euler factor) is defined by
  \[
   \mathcal{E}^{\operatorname{ad}}\left( f_{\alpha}\right)=
   \begin{cases}
    \left( 1-\frac{\beta}{\alpha} \right)\left( 1-\frac{\beta}{p\alpha} \right)
    &\text{if $a = 0$},\\[1.5ex]
    \left(\tfrac{-1}{p-1}\right) \cdot \left(\tfrac{\psi\left( p\right)p^{k-2}}{\alpha^{2}}\right) &\text{if $a = 1$ and $\psi_{p}=\operatorname{id}$}, \\[1.5ex]
    \left( \tfrac{p^{1-a}}{p-1}\right)  \left( \tfrac{\psi\left( p\right)p^{k-2}}{\alpha^{2}} \right)^{a}  G\left( \psi_{p}\right) &\text{otherwise,}
   \end{cases}
  \]
  where $\psi$ is the prime-to-$p$ character of $f$, and $G\left( \psi_{p}\right)$ is the Gauss sum of $\psi_{p}$,
  \item $ \Lambda^{[pN]}(f, g^{*}, l+t)$ is the completed automorphic Rankin--Selberg $L$-function attached to $f\left( \tau\right)$ and $g^{*}\left( \tau\right)\coloneqq \overline{g\left( -\overline{\tau}\right)}$, with Euler factors at primes dividing $pN$ removed (c.f.~\cite[\S 4.1]{LLZ14}; also defined in Section \ref{sec:42}).
 \end{itemize}
\end{letthm}

 Note that our $p$-adic $L$-function does not have a separate ``cyclotomic'' variable; this is not necessary, since (unlike the more familiar ordinary deformation rings) the universal deformation ring $\mathcal{R}^{\operatorname{univ}}$ already incorporates information about twisting by characters: it can be written as the product of a smaller ring $\mathcal{R}^{\text{tame det}}$ (of relative dimension 2 over $\zpa$) parametrising deformations with tamely-ramified determinant, and a copy of the cyclotomic Iwasawa algebra. The differential operator $\theta^t$ appearing in the interpolation formula corresponds to a cyclotomic twist of the Galois representation.

 The proof of the interpolation formula is rather intricate, although the overall strategy is similar to analogous calculations in \cite{loe18} and elsewhere; it will be given at the end of this paper. The key idea is to replace the $p$-depleted Eisenstein series (which varies in $p$-adic families) to another class of Eisenstein series, defined in \cite[\S 3]{kat04}, for which the Rankin--Selberg integral can be computed explicitly. The relation between these two families of Eisenstein series involves the Atkin--Lehner operator at $p$, and the term $\lambda_{p^{b}}(g)\left(\frac{p^{t+1}}{\alpha} \right)^b\frac{P_{p}\left( g,\, p^{t} \alpha^{-1}\right)}{P_{p}\left(g^{*},\, \alpha p^{-l-t}\right)}$ emerges from the effect of this operator on the $q$-expansion coefficients of $f$ and $g$ at $p$; see Lemmas \ref{keylemma} and \ref{lem:correctionterm}.

  \subsubsection*{Compatibility with conjectures}

   In \cite[Conjecture 2.8]{loe21}, the second author made a general conjecture predicting how $p$-adic $L$-functions for families of global Galois representations should behave under specialisations, building on earlier works of \cite{CPR89} and \cite{fk06}. This includes a prediction for the shape of the factors relating the $p$-adic $L$-function to the complex one, involving the Euler factors of the Panchishkin subrepresentation of $V$ and its dual.

When $f$ and $g$ are both crystalline, we verify that our interpolation formula does have the correct Euler factor at $p$, as predicted by the above conjecture. (In the more general case when $a, b > 0$, the factor $\lambda_{p^{b}}( g)\left(\frac{p^{t+1}}{\alpha} \right)^b$ can also be interpreted in these terms -- it is essentially the local $\varepsilon$-factor of this representation -- but we shall not pursue this interpretation here.)

\subsubsection*{Functional equation}
The final goal in this paper is to derive a functional equation for our $p$-adic $L$-function $\mathcal{L}$.
It takes the following form:
\begin{letthm}[Theorem \ref{thm:p-adicfunctionalequation}] We have
 \begin{align*}
  \mathcal{L}= N^{2\left( \boldsymbol{k}_{1}-\boldsymbol{k}_{2} -1 \right)} \gamma\left(\rho_{\mathcal{A}} \right) \mathcal{L}',
 \end{align*}
 where $\boldsymbol{k}_{1}$ and $\boldsymbol{k}_{2}$ are weight characters of $\rho^{\ord}$ and $\rho^{\operatorname{univ}}$ respectively, $\mathcal{L}'$ is a suitably defined dual $p$-adic $L$-function, and $\gamma\left( \rho_{\mathcal{A}}\right)= \prod\limits_{\nu \mid N}\gamma_{\nu}\left( \rho_{\mathcal{A}}\right)$ is the product of local universal $\gamma$-factors at bad primes. \rev{Here, $\rho_{\mathcal{A}}$ is the four-dimensional universal Galois representation that is the focus of our study, constructed as the tensor product of the universal ordinary deformation $\rho^{\ord}$
  and the Tate dual of the universal deformation $(\rho^{\operatorname{univ}})^*(1)$. It is defined precisely in Section 3.3.}
\end{letthm}

 It will be proved as follows: we first consider the dense subset $\Sigma''\left( \mathcal{V},\mathcal{V}^{+}\right) $ of $\operatorname{Spec}\left( \mathbb{T}_{\mathfrak{a}}\right)\times \operatorname{Spec}\left( \mathcal{R}^{\operatorname{univ}}\right)$ consisting of crystalline points (for both $f$ and $g$, of weights as in the preceding Theorem), and construct a $p$-adic $L$-function $\mathcal{L}'$ which approximates the ``dual" of the original $p$-adic $L$-function $\mathcal{L}$ \rev{(on the dense set of crystalline points $\Sigma''\left( \mathcal{V},\mathcal{V}^{+}\right)$, the specializations of $\mathcal{L}'$ are shown to be related via a precise interpolation formula to the complex $L$-values of the dual Rankin-Selberg convolution.)} We then examine the interpolation formula given in Theorem \ref{thm:interpolation formula} (and also the dual interpolation formula given in Theorem \ref{thm:dualinterpolationformula}) and re-interpret the ratio $\frac{P_{p}\left( g, p^{t} \alpha^{-1}\right)}{P_{p}\left(g^{*}, \alpha p^{-l-t}\right)}$ along with the invisible factor $P_{p}\left( f,g^{*},l+t\right)$ as the modified Euler factor at $p$ of \rev{the four-dimensional Rankin-Selberg Galois representation $\rho_{\mathcal{A}}$ evaluated at that point}. This calculation also enables us to deduce the the modified Euler factors at $p$ of $\mathcal{L}$ and $\mathcal{L}'$ are equal. Thus to relate the interpolation formula of $\mathcal{L}$ to that of $\mathcal{L}'$, we need to appeal to a $N$-depleted version of the complex functional equation. This is discussed in Section \ref{sec:complexfunctionalequation}, where we prove that they are related via the product of $\gamma$-factors at bad primes. This motivates our seek for a universal $\gamma$-factor $\gamma_{\nu}\left( \rho_{\mathcal{A}}\right)$ at each bad prime $\nu$, which interpolates the classical ones.

This turns out to be possible, and the key input is the universal $\gamma$-factor $\gamma_{\nu}\left( \rho_{\mathcal{A}},X\right)$ Helm and Moss attached to the representation $\rho_{\mathcal{A}}$ in \cite{hm17}. More precisely, Helm and Moss constructed an element $\gamma_{\nu}\left( \rho_{\mathcal{A}},X\right)$ in the localised ring $\mathcal{T}^{-1}\mathcal{A}[X,X^{-1}]$, for $X$ an indeterminate, $\mathcal{T}$ the set of Laurent polynomials with leading and trailing coefficients units in $\mathcal{A}$, and proved in Theorem 1.1 of \textit{loc. cit.} that this universal $\gamma$-factor does interpolate the classical ones upon specialising. \rev{Here, 
$\mathcal{A}:=\mathbb{T}_{\mathfrak{a}}\hat{\otimes}\mathcal{R}^{\operatorname{univ}}$
  is the half-ordinary Rankin-Selberg universal deformation ring that serves as the base ring for the $p$-adic $L$-function (see Section 3.3).} We have to be careful, however, as in our applications we will take $X$ to be $1$. This may cause trouble, as upon specialising $X$ to $1$, the image of the set $\mathcal{T}$ may contain zero-divisors, in which case the image of $\mathcal{T}^{-1}\mathcal{A}[X,X^{-1}]$ would just be the zero ring, and our definition of $\gamma_{\nu}\left( \rho_{\mathcal{A}}\right)\coloneqq \gamma_{\nu}\left( \rho_{\mathcal{A}},1\right)$ will be meaningless. We will prove the well-definedness of $\gamma_{\nu}\left( \rho_{\mathcal{A}}\right)$ by factorising the ring $\mathcal{A}$ into three pieces $\mathcal{A} = \mathbb{T}_{\mathfrak{a}} \hat{\otimes}_{\mathcal{O}} \mathcal{R}^{\operatorname{tame \ det}} \hat{\otimes}_{\mathcal{O}} \Lambda =: \widetilde{\mathcal{A}}\hat{\otimes}\Lambda$ (where $\mathcal{R}^{\operatorname{tame \ det}}$ is the quotient of $\mathcal{R}^{\operatorname{univ}}$ parameterising deformations of $\overline{\rho}_{2}$ with tame determinants) and relate the universal $\gamma$-factor $\gamma_{\nu}\left( \rho_{\mathcal{A}},X\right)$ to the universal $\gamma$-factor $\gamma_{\nu}\left( \rho_{\widetilde{\mathcal{A}}},X\right)$ of the associated universal deformation $\rho_{\widetilde{\mathcal{A}}}:G_{S} \rightarrow \operatorname{GL}_{4}\left( \mathbb{T}_{\mathfrak{a}} \hat{\otimes}_{\mathcal{O}} \mathcal{R}^{\operatorname{tame \ det}}\right)$.

\begin{revised}
\subsubsection*{Comparison with functional equations using $\varepsilon$-factors}
A key challenge in formulating a $p$-adic functional equation is the interpolation of local constants. This reflects the two possible formulations of the functional equation for classical complex $L$-functions: one can either work with \emph{primitive} $L$-functions (with the optimal local factors at all finite primes, including those where the representation is ramified), so that the functional equation is expressed in terms of $\varepsilon$-factors; or one can work with the \emph{depleted} $L$-functions (omitting the local factors at the ramified finite places), so that the functional equation is expressed in terms of Tate's $\gamma$-factors. The relation between the two formulations is given by the formula
\[\tag{\dag}
\gamma_{\nu}\left(f,g,s,\vartheta \right)= \varepsilon_{\nu}\left(f,g,s, \vartheta \right)\frac{L_{\nu}\left( f^{*},g^{*},k+l-1-s\right)}{L_{\nu}\left(f,g,s \right)}.
\]

The approach used by Chen and Hsieh in \cite{CH18} is to construct $p$-adic $L$-functions interpolating primitive $L$-functions. With this approach, in order to formulate a $p$-adic functional equation, it is necessary to construct a $p$-adic interpolation of $\varepsilon$-factors. This is possible for \textit{primitive} Hida families, where, due to the rigidity of automorphic types result proved in Lemma 2.14 of \cite{FO12}, the rank of the space of inertia invariants is constant across all classical specializations, for each fixed prime dividing the level. However, in the more general context of a universal deformation space, this rigidity is lost; it is possible to have non-trivial intersections between components that are generically Steinberg at some prime $\ell \ne p$ and those which are generically principal-series at $\ell$, and at such intersection points the rank of the inertia invariants is not locally constant. This prevents the direct interpolation of standard $\varepsilon$-factors: while it is possible to interpolate Deligne's \textit{modified} $\varepsilon_0$-factors as in \cite{yas09}, the relation between these and the $\varepsilon$-factors involves a term depending on the inertia invariants, and this does not extend to an analytic function on the deformation space.

Instead, we work with depleted $L$-functions, which are better-behaved under congruences than primitive ones (as has also been observed in a number of other recent works, such as \cite{RSV23} and \cite{Del24}). Hence the universal $\gamma$-factor of Helm and Moss \cite{hm17} provides the most natural framework for our construction, showing that the ratio of three terms on the right-hand side of ($\dag$) interpolates over the deformation space, whereas the three individual terms in general do not.
\end{revised}

\subsection{Acknowledgements} This paper forms part of the first author's Warwick PhD thesis, under the supervision of the second author. He would like to thank the second author for his superb guidance and insightful conversations. He also gratefully acknowledges Olivier Fouquet's careful reading of his thesis and stimulating conversations in his viva, which has led to many improvements on both contents and presentation of this paper. During the outbreak of COVID-19, the first author stayed in BICMR of Peking University, and he would like to thank his hosts, Ruochuan Liu and Liang Xiao, for their hospitality. We thank Pak-Hin Lee, Patrick Allen, James Newton and Ju-Feng Wu for helpful communications. We are also grateful to the referee for his/her very careful reading of our manuscript and valuable suggestions.

\subsection{Notations and conventions}

Throughout this paper we fix a prime $p > 2$. We also fix an embedding $\iota_{p}: \overline{\QQ} \hookrightarrow \overline{\QQ_{p}}$, and the symbol $\iota_{p}\left( \cdot \right)$ is often omitted if there is no danger of confusion. We fix a $p$-adic norm $\mid \cdot \mid_{p}$ on $\overline{\QQ_{p}}$ such that $|p|=p^{-1}$.

Let the action of $\operatorname{GL}_{2}^{+}(\QQ)$ on a weight $k$ modular form $f$ be given by
\[ \left( f \mid_{k} \begin{psmallmatrix}
a & b \\ c & d
\end{psmallmatrix}\right)\left( \tau\right) \coloneqq \left( ad-bc\right)^{k-1}\left( c \tau+d\right)^{-k}f\left( \frac{a\tau +b}{c \tau  +d}\right).\]
For integers $Q,N \geq 1$, we write $Q \ediv N$ to signify that $Q \mid N$ and $\left( Q, \frac{N}{Q} \right)= 1$. In this case, define the Atkin--Lehner operator $W_Q$ at $Q$ by any matrix of the form $\begin{psmallmatrix}
Qx & y \\ NQz & Qw
\end{psmallmatrix}$, with $x,y,z,w$ integers such that $Qxw-Nyz=1$, $Qx=w=1 \pmod{N}$, $y=-1 \pmod{Q}$ and $Nz=1 \pmod{Q}$. This is the convention taken in \cite[\S 2.5]{klz17} \footnote{Note this differs from the original Atkin--Lehner operator considered in \cite{al78}. More precisely, if we let $W_{Q}^{\operatorname{AL}}$ denote the Atkin--Lehner operator defined in \cite{al78} (i.e.\ $y \equiv 1 \pmod{Q}$ and $x \equiv 1 \pmod{N/Q}$), then $W_{Q}^{\operatorname{AL}}=W_{Q}\langle -1\rangle_{Q}\langle Q^{-1}\rangle_{N/Q}$.}, where several useful properties of these operators were also discussed. For $\Gamma = \Gamma_{1}\left( N\right)$ or $\Gamma=\Gamma_{1}\left( R\left( S\right)\right)\coloneqq \Gamma_{1}\left( R\right) \cap \Gamma_{0}\left( RS\right)$, let $N_{\operatorname{GL}_{2}^{+}\left( \QQ\right)}\left( \Gamma\right)$ denote the normaliser group of $\Gamma$ in $\operatorname{GL}^{+}_{2}\left( \QQ\right)$, and let $\mathcal{G}\coloneqq N_{\operatorname{GL}_{2}^{+}\left( \QQ\right)}\left( \Gamma\right)/ \Gamma$. For $Q \mid\mid N$ and $x \in \left( \mathbb{Z}/Q\mathbb{Z}\right)^{\times}$, let $\langle x\rangle_{Q}$ denote the class in $\mathcal{G}$ of any representative $x \in \operatorname{SL}_{2}\left( \mathbb{Z}\right)$ of the form $x=\begin{psmallmatrix}
    a & b \\
    Nc & d
\end{psmallmatrix}$ with $d \equiv x \pmod{Q}$ and $d \equiv 1 \pmod{N/Q}$.

Let $f\in S_{k}\left( N, \varepsilon \right)$ be a newform of level $N$ and character $\varepsilon$. For $Q \| N$, write $\varepsilon\coloneqq  \varepsilon_{Q} \cdot \varepsilon_{N/Q}$ for characters $\varepsilon_{Q}$ and $\varepsilon_{N/Q}$ modulo $Q$ and $N/Q$, respectively. There exists a unique newform $f \otimes \varepsilon_{Q}^{-1}  \in S_{k}\left( \Gamma_{1}\left( N\right),\overline{\varepsilon_{Q}}\varepsilon_{N/Q}\right)$ and a scalar $\lambda_{Q}\left( f\right) \in \mathbb{C}$ such that $f \mid_{k} W_{Q} = \lambda_{Q}\left( f\right) \cdot f \otimes \varepsilon_{Q}^{-1}$. The scalar $\lambda_{Q}\left( f\right)$ is called the \textit{Atkin--Lehner pseudo-eigenvalue} of $f$ at $Q$. Hereafter the "tensor product" notation $f \otimes \chi$ for an arbitrary character $\chi$ modulo $M$ always means the newform twist, while the subscript $f_{\chi}$ means the ``naive" twist $f_{\chi}\coloneqq  \sum \chi\left( n\right) a_{n}\left( f\right)q^{n}$.
Note that under our conventions, the Atkin--Lehner pseudo-eigenvalues are strictly multiplicative, i.e. if $Q=p^{b}$ is a prime power such that $Q \mathrel{\|} N$, then $\lambda_{N}\left( f\right)=\lambda_{Q}\left( f\right)\lambda_{N/Q}\left( f\right)$.

 We define the $p$-depletion of $f$ to be $f^{[p]}\coloneqq  \sum\limits_{\substack{n \geq 0 \\ p \nmid n}} a_{n}q^{n}$. More generally, for $M=p_{1}^{i_{1}} \ldots p_{r}^{i_{r}}$, we define the $M$-depletion of $f$ to be $f^{[M]}\coloneqq \left(\left( f^{[p_{1}]}\right)^{[p_{2}]}\hdots \right)^{[p_{r}]}=\sum\limits_{\substack{n \geq 0 \\ n \notin \operatorname{Supp}\left( M\right)}}a_{n}q^{n}$. Note that for $Q$ with $\left( Q,M\right)=1$, the Atkin--Lehner operator $W_{Q}$ commutes with the $M$-depletion process.

Finally, we fix our convention for Galois representations as that of \cite[\S 3]{loe21}. For a prime $\nu$, $\operatorname{Frob}_{\nu}$ denotes an arithmetic Frobenius.

\section{Deformation rings and Hecke algebras}

Fix a prime $p > 2$, and let $S_{1}$ and $S_{2}$ be two finite set of primes containing $p$ and $\infty$. Let $L$ be a finite extension of $\QQ_{p}$, $\mathcal{O}$ its ring of integers, with residue field $\mathbb{F}$.
We fix two residual representations $\overline{\rho}_{1}: G_{S_{1}} \to \operatorname{GL}_{2}\left( \mathbb{F}\right)$ and $\overline{\rho}_{2}:G_{S_{2}} \to \operatorname{GL}_{2}\left( \mathbb{F}\right)$. We assume $\overline{\rho}_{i}$ (for $i = 1,2$) satisfies the following properties: \begin{itemize}
    \item $\overline{\rho}_{i}$ is absolutely irreducible.
    \item $\overline{\rho}_{i}$ is odd.
    \item (Taylor--Wiles condition) The restriction of $\overline{\rho}_{i}$ to the absolute Galois group of
    $\QQ(\zeta_p)$ is irreducible.
     \item if $\left.\bar{\rho}_{i}\right|_{G_{\QQ_p}}$ is not irreducible, with semisimplification $\epsilon_{1, p} \oplus \epsilon_{2, p}$, then we have $\epsilon_{1, p} / \epsilon_{2, p} \notin\left\{1, \overline{\varepsilon}_{\operatorname{cyc}}^{\pm 1}\right\}$, where $\overline{\varepsilon}_{\operatorname{cyc}}$ is the $\bmod$ $ p$ cyclotomic character.
\end{itemize}

In addition, we assume $\overline{\rho}_{1}$ is ordinary at $p$ (but we make no such assumption on $\overline{\rho}_{2}$). Then by \cite{maz89}, the functor representing deformations (resp. ordinary deformations\footnote{More precisely: we fix a \emph{choice} of a local unramified subrepresentation in $\overline{\rho}_{1}$, and we study deformations with a local unramified subrepresentation lifting this choice. It can happen that $\overline{\rho}_{1}$ is unramified at $p$, but our running assumptions imply that in this case it is isomorphic to the direct sum of two \emph{distinct} unramified characters. So we always have either 1 or 2 choices for the unramified subrepresentation.}) of $\overline{\rho}_{2}$ (resp. $\overline{\rho}_{1}$) is representable, and we denote by $\left( \rho^{\operatorname{univ}}, \mathcal{R}^{\operatorname{univ}} \right)$ (resp. $\left( \rho^{\ord}, \mathcal{R}^{\ord}\right) $) the universal pair of this functor.

As in \cite{boe03,epw05,gou90}, we can attach to $\rho^{\ord}$ and $\rho^{\operatorname{univ}}$ ``tame conductors". More precisely, for a residual representation $\overline{\rho}:G_{S} \rightarrow \operatorname{GL}_{2}\left( \mathbb{F}\right)$, let $N(\overline{\rho})\coloneqq  \prod\limits_{l \neq p}l^{n(l,\overline{\rho})}$ be the tame conductor of $\overline{\rho}$, as defined in \cite[\S 3]{gou90}. We define the\textit{ universal tame conductor} $N_{S}$ of $\overline{\rho}$ by $N_{S} \coloneqq \prod\limits_{l \neq p} l^{n_{S}(l)}$, where the exponents are determined by the following rule:
\begin{enumerate}
 \item If $\bar{\rho}$ is unramified, then $n_{S}(l)=2$ if $l \in S$ and is $0$ otherwise.
 \item If $\bar{\rho}^{I_\ell}$ is 1-dimensional (where $I_\ell$ is the inertia group), then $n_{S}(l)=n(l,\bar{\rho})+1$.
 \item If $\bar{\rho}^{I_\ell} = 0$, then $n_{S}(l)=n(l,\bar{\rho})$.
\end{enumerate}

For $\overline{\rho}=\overline{\rho}_{1},\overline{\rho}_{2}$, define $N_{1}'\coloneqq  N_{S_{1}}$ and $N_{2}\coloneqq N_{S_{2}}$. Let $S\left( N_{2}, \mathcal{O} \right)$ denote the space of $p$-adic modular forms over $\mathcal{O}$ of tame level $N_{2}$, \rev{obtained as the direct limit $\varinjlim\limits_{r}S_2(N_2p^r,\mathcal{O})$,} and let $\mathbb{T}_{N_2}$ denote the subspace of $\operatorname{End}_{\mathcal{O}}\left(S\left( N_{2},\mathcal{O}\right) \right)$ generated by Hecke operators $T_{l}$ for all $l \nmid pN_{2}$ and the diamond operators \rev{(naturally arising as the inverse limit of Hecke algebras acting on spaces of finite-level modular forms by duality)}. Since by assumption $\overline{\rho}_{2}$ is odd, it arises from modular forms, and the duality between modular forms and Hecke algebras determines a maximal ideal $\mathfrak{n}$ of $\mathbb{T}_{N_2}$, corresponding to $\overline{\rho}_{2}$. Let $\mathbb{T}_{\mathfrak{n}}$ denote the completion of $\mathbb{T}_{N_2}$ at $\mathfrak{n}$. Then we have the following ``$\mathcal{R}=\mathbb{T}''$ theorem:

\begin{thm} \label{r=t}
   Under the running assumptions on $\overline{\rho}_{2}$, we have $\mathcal{R}^{\operatorname{univ}} \cong \mathbb{T}_\mathfrak{n}$.
\end{thm}

\begin{proof}
\rev{
    Under a somewhat stronger hypothesis on $\overline{\rho}_2$ (assuming $\overline{\rho}_{2} \mid_{ G_{\QQ_{p}}}$ has a twist that is either ordinary, or irreducible and flat), this is Theorem 3.9 of \cite{boe03}. In the general setting above, the proof is given in Section 7.3 of \cite{emer06}. (Emerton does not formulate his results in precisely this form, but he describes a generalization of \cite[Corollary 3.8]{boe03} to this setting, and the argument deducing Theorem 3.9 of \textit{op.cit.} from this extends without change, giving the equality $\mathcal{R}^{\operatorname{univ}}=\mathbb{T}_{\mathfrak{n}}$ under Emerton's hypotheses\footnote{We are grateful to Patrick Allen for his explanations on this.}.)}
\end{proof}

Let $e$ denote Hida's ordinary projector, and
let $\mathbb{T}^{\ord}_{N_{1}'}$ denote the subalgebra of $\operatorname{End}_{\mathcal{O}}\left(eS\left(N_{1}', \mathcal{O} \right) \right)$ generated by the Hecke operators $T_{l}$ for $l \nmid pN_{1}'$, $U_{p}$ and the diamond operators. As before, the residual representation $\overline{\rho}_{1}$ determines a maximal ideal $\mathfrak{m}$ of $\mathbb{T}^{\ord}_{N_{1}'}$, and we denote by $\mathbb{T}^{\ord}_{\mathfrak{m}}$ the corresponding completion. Analogously, we have

\begin{thm}[\cite{boe03}]
    Under the running assumptions on $\overline{\rho}_{1}$, we have $\mathcal{R}^{\ord} \cong \mathbb{T}^{\ord}_{\mathfrak{m}}$.
\end{thm}

\begin{defi}[classical and nearly classical points]
Let $f$ be a normalised eigenform (resp. normalised ordinary eigenform) of tame level $N_2$ (resp. $N_{1}'$) . If $\rho_f$ is a deformation of $\overline{\rho}_{2}$ (resp. $\overline{\rho}_{1}$), then it determines a $\overline{\QQ_p}$-point of $\operatorname{Spec}(\mathcal{R^{\operatorname{univ}}})$ (resp. $\operatorname{Spec}\left( \mathcal{R}^{\ord}\right)$). Such points are called \emph{classical points} of $\mathcal{R}^{\operatorname{univ}}$ (resp. $\mathcal{R}^{\ord}$).

Moreover, if the corresponding Galois representation of a $\overline{\QQ_p}$-point of $\operatorname{Spec}(\mathcal{R}^{\operatorname{univ}})$ is of the form $\rho_f \otimes \varepsilon_{\operatorname{cyc}}^{-t}$ for a normalised eigenform $f$ and an integer $t$, then it is called a \emph{nearly-classical point}.

\end{defi}

As explained in \cite[\S 3]{loe21}, for $t \geq 0$, the Galois representation $\rho_{g} \otimes \varepsilon_{\operatorname{cyc}}^{-t}$ corresponds to the $p$-adic modular form $\theta^{t}\left( g\right)$, where $\theta = q \dfrac{d}{dq}$ is the Serre-Tate differential operator.

Since the prime-to-$p$ conductor of $\rho^{\mathrm{univ}}$ is bounded by $N_2$, we can find characters $\epsilon : \left(\mathbb{Z}/ N_{2}\mathbb{Z}\right)^{\times} \to (\mathcal{R}^{\mathrm{univ}})^\times$, and $\mathbf{k}_2 : \zpm \to (\mathcal{R}^{\mathrm{univ}})^\times$, such that
\[ \operatorname{det} \rho^{\operatorname{univ}}=\epsilon \cdot \varepsilon_{\operatorname{cyc}}^{1-\boldsymbol{k}_{2}}, \]
where as usual $\varepsilon_{\operatorname{cyc}}$ is the $p$-adic cyclotomic character. We call $\boldsymbol{k}_{2}$ the \textit{universal weight-character} of $\rho^{\operatorname{univ}}$. Completely analogously, we can write $\operatorname{det}\left( \rho^{\ord}\right)$ as $\psi' \cdot \varepsilon_{\operatorname{cyc}}^{1-\boldsymbol{k}'_{1}}$ for a character $\psi' : \left(\mathbb{Z}/ N_{1}'\mathbb{Z}\right)^{\times} \to (\mathbb{T}^{\ord}_{\mathfrak{m}})^\times$ and a universal weight-character $\boldsymbol{k}'_{1} : \zpm \to (\mathbb{T}^{\ord}_{\mathfrak{m}})^\times$.

\section{The \texorpdfstring{$p$}{p}-adic \texorpdfstring{$L$}{L}-function}

\subsection{The Hida family associated to $\overline{\rho}_{1}$}

 As above, we identify the universal ordinary deformation ring $\mathcal{R}^{\ord}$ with the restricted Hecke algebra $\mathbb{T}^{\ord}_{\mathfrak{m}}$. The ring $\mathbb{T}^{\ord}_{\mathfrak{m}}$ is Noetherian, and hence has finitely many minimal prime ideals. We choose a minimal prime $\mathfrak{a}$, corresponding to an irreducible component of $\operatorname{Spec}\mathbb{T}^{\ord}_{\mathfrak{m}}$ (a ``branch'' in the terminology of \cite{klz17}). This determines a primitive Hida family of some level $N_1 \mid N_1'$, as follows.

 For a prime-to-$p$ positive integer $M$, we let $\widetilde{\mathbb{T}}^{\ord}_M \supseteq \mathbb{T}^{\ord}_M$ denote the full Hecke algebra acting on the space of ordinary $p$-adic modular forms over $\mathcal{O}$ of tame level $M$ (including the Hecke operators $U_\ell$ for $\ell \mid M$). Following \cite[\S 2]{epw05}, we define $\widetilde{\mathbb{T}}^{\mathrm{new}}_{M}$ (the ``new quotient'') to be the quotient of $\widetilde{\mathbb{T}}^{\ord}_{M}$ that acts faithfully on the space of newforms at level $M$ (see Theorem 2.1.3 of \textit{loc.cit.} for properties of this algebra).

 As in \cite[\S 2.5]{epw05}, we can find a unique divisor $N\left( \mathfrak{a}\right)$ of $N'_{1}$, and a unique minimal prime ideal $\mathfrak{a}'$ of $\widetilde{\mathbb{T}}^{\operatorname{new}}_{N\left( \mathfrak{a}\right)}$, such that there exists a natural embedding map of local domains $\mathbb{T}^{\ord}_\mathfrak{m}/\mathfrak{a} \hookrightarrow \widetilde{\mathbb{T}}^{\operatorname{new}}_{N\left( \mathfrak{a}\right)}/\mathfrak{a}'$, and the representation $\rho\left( \mathfrak{a}\right):G_{\QQ}\to \operatorname{GL}_{2}\left( \mathbb{T}^{\ord}_{\mathfrak{m}}/\mathfrak{a}\right)$ induced from $\rho^{\ord} :G_{\QQ} \to \operatorname{GL}_{2}\left( \mathbb{T}^{\ord}_{\mathfrak{m}}\right)$ via the natural map has tame conductor $N\left( \mathfrak{a}\right)$ (see Proposition 2.5.2, Remark 2.5.4 and Corollary 2.5.5 of \textit{op. cit.}). We now define $N_{1}\coloneqq N\left( \mathfrak{a}\right)$.

Let $\mathbb{T}_{\mathfrak{a}}$ denote the integral closure of the domain $\widetilde{\mathbb{T}}^{\operatorname{new}}_{N_{1}}/\mathfrak{a}'$. So $\mathbb{T}_{\mathfrak{a}}$ is a normal domain, finite flat over $\Lambda$ (see \cite[Proposition 2.2.3]{epw05} for a proof). As in \cite[\S 2.7]{epw05}, we can attach to $\rho\left( \mathfrak{a}\right)$ a $\Lambda$-adic form $\mathcal{F}\left( q\right)$ by defining:
$$
\mathcal{F}\left( q\right)\coloneqq  \sum\limits_{n \geq 1} \left(T\left( n\right) \bmod{ \mathfrak{a}'} \right)q^{n} \in \mathbb{T}_{\mathfrak{a}}[[q]] .
$$ At each classical point, $\mathcal{F}\left( q\right)$ specialises to an ordinary $p$-stabilised newform of tame level $N_{1}$, and we may view $\mathcal{F}\left( q\right)$ as a primitive Hida family.

Composing the maps $\mathbf{k}_1'$ and $\psi'$ above with the natural map $\mathbb{T}^{\ord}_{\mathfrak{m}} \to \mathbb{T}_\mathfrak{a}$, we obtain a weight-character $\mathbf{k}_1$ and prime-to-$p$ character $\psi$ valued in $\mathbb{T}_\mathfrak{a}^\times$. These are the weight-character and the prime-to-$p$ nebentype of the primitive Hida family $\mathcal{F}$.

Consider the twisted $\Lambda$-adic form:
$
\mathcal{F}_{\psi^{-1}}\left( q\right)= \sum\limits_{n \geq 1}\psi^{-1}\left( n\right)a_{n}\left( \mathcal{F}\right)q^{n}.
$
Let $\mathcal{F}^{c}$ denote the primitive Hida family associated to this depleted $\Lambda$-adic form; its Fourier coefficients are given explicitly in \cite[p. 18]{Hsi17}. In particular, for every classical point $Q$, if $\mathcal{F}$ specialises to a $p$-stabilised newform $f$ modulo $Q$, then $\mathcal{F}^{c}$ specialises to $f^{c}$ modulo $Q$, where $f^{c}$ denotes the unique $p$-stabilised newform corresponding to $f \mid W_{N_{1}}$.

To this primitive Hida family, \cite[\S 7.7]{klz17} attached a fractional ideal $I_{\mathfrak{a}} \subset \operatorname{Frac} \mathbb{T}_{\mathfrak{a}}$, \footnote{Hsieh proved in \cite[p. 18]{Hsi17} that the congruence ideal attached to $\mathcal{F}^{c}$ is the same as that of $\mathcal{F}$.} and a unique linear functional $\lambda_{\mathcal{F}^{c}}' : e\mathcal{S}(N_{1}, \Lambda) \otimes_{\Lambda} \mathbb{T}_{\mathfrak{a}} \rightarrow I_{\mathfrak{a}}^{-1}$, characterised by mapping $\mathcal{F}^{c}$ to $1$.

Let $N$ be a positive integer divisible by $N_{1}$ and $N_{2}$, and with the same prime factors as $N_1 N_2$. It will be convenient also to assume that for each prime $\ell$ which divides $N_1$ but not $N_2$, we have $\ell^2 \mid N$. Let $\operatorname{Tr}_{N_{1}}^{N}$ denote the trace map from level $N$ forms to level $N_{1}$ forms (see e.g.\ \cite[\S 1]{Hid88} for a definition). Then we define the ``level $N$" linear functional $\lambda_{\mathcal{F}}$ as
$$
\lambda_{\mathcal{F}^{c}}\coloneqq  \lambda'_{\mathcal{F}^{c}}  \circ \operatorname{Tr}_{N_{1}}^{N}.
$$



\subsection{The universal deformation family associated to $\overline{\rho}_{2}$} \label{sec:universaleigenform}

In this Subsection, we define the (depleted) universal eigenform associated to the universal deformation $\rho^{\operatorname{univ}}$, using reciprocal local Euler factors at good primes. This will be called \textit{the universal deformation family} associated to $\overline{\rho}_{2}$.

\begin{defi}
 Let $S(pN)$ denote the set of integers which are not coprime to $pN$. Then we define the universal $pN$-depleted eigenform associated to $\overline{\rho}_{2}$ by
 \begin{equation*}
  \mathcal{G}^{[pN]} \coloneqq  \sum\limits_{n \notin S(pN)} t_n q^n,
 \end{equation*}
 where the sequence $\{t_n \}_{n \notin S(pN)}$ is determined by the following identity of formal Dirichlet series:
 \begin{equation*}
  \sum\limits_{n \notin S(pN)}t_n n^{-s} = \prod\limits_{\ell \nmid pN} \operatorname{det}(1-\ell^{-s}\rho^{\operatorname{univ}}\left( \operatorname{Frob}^{-1}_{\ell})\right)^{-1}.
 \end{equation*}
\end{defi}

\begin{prop}[Gouvea, Loeffler]
 The $q$-expansion $\mathcal{G}^{[pN]}$ is a $p$-adic modular form, with coefficients in $\mathcal{R}^{\operatorname{univ}}$, of tame level $N$, weight character $\boldsymbol{k}_{2}$ and prime-to-$p$ character $\epsilon$, and is a normalised eigenform for all Hecke operators.
\end{prop}

\begin{proof}
 We first give the proof supposing $N = N_2$. Let $\widetilde{\mathbb{T}}_{N_2}$ denote the \emph{full} Hecke algebra acting on $p$-adic modular forms of tame level $N_2$ (including the Hecke operators at the bad primes). As in \cite[Prop.~2.4.2]{epw05} in the ordinary case, we can find a maximal ideal $\tilde{\mathfrak{n}}$ of $\widetilde{\mathbb{T}}_{N_2}$ lying above $\mathfrak{n}$, with the property that the natural map
 $\mathbb{T}_{\mathfrak{n}} \to (\widetilde{\mathbb{T}}_{N_2})_{\tilde{\mathfrak{n}}}$
  is an isomorphism, and the $U_\ell$ for $\ell \mid N_2$ map to 0 in $\left(\widetilde{\mathbb{T}}_{N_2}\right)_{\tilde{\mathfrak{n}}}$.
   Arguing as in Theorem 3.10 of \cite{loe21}, using the duality between $p$-adic modular forms and Hecke algebras proved in \cite[\S3]{gou88}, we deduce that the formal power series $\mathcal{G}^{[pN_2]}$ is a $p$-adic modular form as required.

 If $N \ne N_2$, then $\mathcal{G}^{[pN]}$ is given by $\ell$-depleting $\mathcal{G}^{[pN_2]}$ for each prime dividing $N_1$, but not $N_2$. Since the $\ell$-depletion raises the tame level by a factor of $\ell^2$, and $\ell^2 \mid N$, the form $\mathcal{G}^{[pN]}$ does indeed have tame level $N$.
\end{proof}

At each nearly classical point $\theta^{t}\left( g\right)$, the universal eigenform $\mathcal{G}^{[pN]}$ specialises to $\theta^{t}\left( g^{[pN]}\right)$. Define $$
\mathcal{G}^{[pN]}_{\epsilon^{-1}}\coloneqq  \sum\limits_{n \notin S\left( pN\right)} \epsilon\left( n\right)^{-1}t_{n}q^{n}.
$$

\subsection{Construction of the $p$-adic $L$-function} \label{sec:53}
We will be interested in Galois representations for Rankin--Selberg convolutions of modular forms. To that end, we define $\mathcal{A}= \mathbb{T}_{\mathfrak{a}} \hat{\otimes}_{\mathcal{O}} \mathcal{R}^{\operatorname{univ}}$, and call it the \textit{half-ordinary Rankin--Selberg universal deformation ring}. It has relative dimension four over $\mathcal{O}$. The universal weight characters $\boldsymbol{k}_{1}:\zpm \rightarrow \mathbb{T}_{\mathfrak{a}}^{\times}$ and $\boldsymbol{k}_{2}:\zpm \rightarrow \left(\mathcal{R}^{\operatorname{univ}} \right)^{\times}$ naturally extend to characters over $\mathcal{A}$, and we may view $\mathcal{G}^{[pN]}_{\epsilon^{-1}}$ as a $p$-adic eigenform with coefficients in $\mathcal{A}$ and weight character $\boldsymbol{k}_{2}$, by base extension.

Consider the representation $\rho_{\mathcal{A}} \coloneqq  \rho^{\ord} \otimes \left( \rho^{\operatorname{univ}}\right)^{*}\left( 1\right)$, where as usual, $\left( \ \cdot \ \right)^{*}\left( 1\right)$ denotes the Tate dual representation. It is a rank four Galois representation from the Galois group $G_{S}$ to $\operatorname{GL}_{4}\left( \mathcal{A}\right)$, and we denote the underlying module by $\mathcal{V}$. We call $\rho_{\mathcal{A}}$ (slightly abusively, also $\mathcal{V}$) the \textit{half-ordinary Rankin--Selberg universal deformation}. Since $\rho^{\ord}$ is ordinary, $\rho^{\ord}\mid_{G_{\QQ_{p}}}$ has a one-dimensional unramified subrepresentation, which we denote by $\left( \rho^{\ord}\right)^{+}$. As in \cite{loe21}, we define $\mathcal{V}^{+}\coloneqq \left( \rho^{\ord}\right)^{+} \otimes \left( \rho^{\operatorname{univ}}\right)^{*}\left( 1\right)$, which in \textit{loc. cit. } was called a \textit{Panchishkin subfamily} of $\mathcal{V}$. It is a rank two local subrepresentation of $\mathcal{V}$.

Define $\Sigma\left( \mathcal{V},\mathcal{V}^{+}\right) \subset \operatorname{Spec}\left( \mathbb{T}_{\mathfrak{a}}\right) \times \operatorname{Spec}\left( \mathcal{R}^{\operatorname{univ}}\right) $ to be the subset consisting of all modular points of the form $\left( f, \theta^t \left( g \right) \right)$, such that the weight $k$ of $f$ satisfies $k \geq 2$, the weight $l$ of $g$ satisfies $l \geq 1$, and $t$ is an integer between $0$ and  $k-l-1$. This will be the range of interpolation for our $p$-adic $L$-function.



 Let $\zeta_N$ be a primitive $N$-th root of unity, and we identify it with its image $\iota_{p}(\zeta_N)$ in $\mathbb{C}_p$. Enlarging $\mathcal{O}$ if necessary, we assume $\zeta_N \in \mathcal{O}$.

\begin{defi}
For a character $\boldsymbol{k}:\mathbb{Z}^{\times}_{p} \rightarrow \mathcal{A}^{\times}$, let $\boldsymbol{F^{[p]}_{k}} \coloneqq  \sum\limits_{p \nmid n} q^{n} (\sum\limits_{d \mid n} \left( n/d \right)^{\boldsymbol{k}-1}(\zeta_N^d + (-1)^{\boldsymbol{k}}\zeta_N^{-d} ) ) \in \mathcal{A}[[q]]
$ denote the $p$-adic family of Eisenstein series of weight character $\boldsymbol{k}$. We define the $p$-adic Rankin--Selberg $L$-function $\mathcal{L}$ by \begin{equation*}
    \mathcal{L} \coloneqq  \lambda_{\mathcal{F}^{c}}\left( e \left(  \mathcal{G}_{\epsilon^{-1}}^{[pN]} \cdot \boldsymbol{F^{[p]}_{k_{1}-k_{2}}}  \right) \right)  \in I_{\mathfrak{a}}^{-1} \otimes_{\mathbb{T}_{\mathfrak{a}}}\mathcal{A}.
\end{equation*}

\end{defi}


Given a pair $\left( f, \theta^{t}g \right) \in \Sigma\left( \mathcal{V},\mathcal{V}^{+}\right)$, we denote by $\alpha$  the unique ordinary  root of the Hecke polynomial of $f$ at $p$, and we denote by $f_{\alpha}$ the $p$-stabilisation of $f$ with $U_{p}$ eigenvalue $\alpha$.
Let $f_{\alpha}^{c}$ denote the unique ordinary $p$-stabilised newform associated to $f_{\alpha}$ satisfying $a_{n}\left( f_{\alpha}^{c}\right)= \psi\left( n\right)^{-1}a_{n}\left( f\right) $ for all $n$ with $\left(n,N \right)=1$, and
let $\lambda_{f_{\alpha}^{c}}$ be the specialisation of $\lambda_{\mathcal{F}^{c}}$ at $f$,
characterised by mapping $f_{\alpha}^{c}$ to $1$. Then by definition, \begin{equation*}
    \mathcal{L}\left( f,\theta^{t}\left( g\right) \right)= \lambda_{f_{\alpha}^{c}}\left( \theta^{t}\left(g_{\epsilon^{-1}}^{[pN]} \right) \cdot F^{[p]}_{k-l-2t,\psi_{p}\epsilon_{p}^{-1} } \right),
\end{equation*} where $\psi_{p}$ (resp.$\epsilon_{p}$) is the $p$-part of the character of $f$ (resp. $g$).

\subsection{The interpolation formula} \label{sec:42}

 For $\ell \neq p$ a prime, let \begin{equation*}
    P_{\ell}(f,g,X) = \operatorname{det}\left( 1-X\operatorname{Frob}^{-1}_{\ell} \mid \left( \rho_{f} \otimes \rho_{g} \right)^{I_{\ell}} \right)
\end{equation*} denote the local Euler factor at $\ell$ as defined in \cite[Definition 4.1.1]{LLZ14} (see also \cite{Jac72} for a definition using automorphic terms), and define $$
P_{p}\left( f,g,X\right)= \operatorname{det}\left( 1- X\varphi:\boldsymbol{D}_{\operatorname{cris}}\left( \rho_{f} \otimes \rho_{g}\right)\right),
$$ where $\boldsymbol{D}_{\operatorname{cris}}$ denotes the crystalline functor, and $\varphi$ is the crystalline Frobenius (introduced by Fontaine \cite{fon94}, see also \cite{Faltings1987,Scholl1990, Tsuji1999padicC,Saito1997ModularFA} for the theory applied to modular forms).  We define the following Rankin--Selberg $L$-functions
\begin{defi} \label{defi:complexrankinselberglfunction}
    \begin{align*}
    L_{\left( pN\right)}\left( f,g,s\right)&= \prod\limits_{\ell \nmid pN} P_{\ell}\left( f,g, \ell^{-s}\right)^{-1}, \\
    L\left( f,g,s\right)&= \prod\limits_{\ell } P_{\ell}\left( f,g, \ell^{-s}\right)^{-1}.
\end{align*}
Let $l$ be the weight of $g$, and define $\Gamma_{\mathbb{C}}\coloneqq \left( 2\pi\right)^{-s}\Gamma\left( s\right)$. Then we define the completed $L$-functions
  \begin{align*}
    \Lambda^{[pN]}\left( f,g,s\right)&= \Gamma_{\mathbb{C}}(s)\Gamma_{\mathbb{C}}(s-l+1)L_{\left( pN\right)}\left( f,g,s\right),\\
    \Lambda\left( f,g,s\right)&= \Gamma_{\mathbb{C}}(s)\Gamma_{\mathbb{C}}(s-l+1)L\left( f,g,s\right).
\end{align*}
\end{defi}

\begin{thm} \label{thm:interpolation formula}
For all $\left( f, \theta^t \left( g \right) \right) \in \Sigma\left( \mathcal{V},\mathcal{V}^{+}\right)$, we have
\begin{multline*}
    \mathcal{L}\left( f,\theta^{t}\left( g\right) \right) = i^{k-l-2t}2^{1-k}N^{2+2t-k+l}\psi^{-1}_{p}\left( N\right) \epsilon_{p}\left( -N \right)    \Lambda^{[pN]}(f, g^{*}, l+t) \\
    \times \lambda_{p^{b}}( g)\left(\frac{p^{t+1}}{\alpha} \right)^b\frac{P_{p}\left( g, p^{t} \alpha^{-1}\right)}{P_{p}\left(g^{*}, \alpha p^{-l-t}\right)\mathcal{E}^{\operatorname{ad}}\left( f_{\alpha}\right)\langle f, f \rangle_{N_{1}p^{a}}}, 
\end{multline*} \rev{where $a$ (resp. $b$) is the power of $p$ dividing the level at which $f$ (resp. $g$) is new, $\lambda_{p^{b}}\left( g\right)$ is the Atkin--Lehner pseudo-eigenvalue of $g$ at $p^{b}$, $\psi_{p}$ (resp. $\epsilon_{p}$) is the $p$-part character of $f$ (resp. $g$),$\alpha$ is the unique ordinary root of the Hecke polynomial of $f$ at $p$,$P_{p}\left( g,X\right)$ is the polynomial satisfying $P_{p}\left(g,X \right)^{-1}=\sum\limits_{u \geq 0} X^{u}a_{p^{u}}\left( g\right)$, and $\mathcal{E}^{\operatorname{ad}}\left( f_{\alpha}\right)$ is the adjoint Euler factor defined in Theorem A.} 
\end{thm}

The proof of this Theorem will be given at the end of this paper. We remark here that this $p$-adic $L$-function is necessarily unique, due to the density of modular points in the universal deformation spaces.
\subsection{The dual construction}
We start working towards the functional equation. As outlined in the introduction, we will define a dual $p$-adic $L$-function $\mathcal{L}'$, which interpolates $\Lambda^{[pN]}( f^{*}, g, k-1-t)$ upon specialising at $\left( f, \theta^t \left( g \right) \right)$, so that we can use the complex functional equation of $\Lambda( f^{*}, g, k-1-t)$ to derive a candidate for the $p$-adic functional equation using the interpolation formulae for $\mathcal{L}$ and $\mathcal{L}'$. \footnote{\rev{Recall the conjugate modular form $f^*$ of $f$ is defined by $f^*(\tau):=\overline{f(-\overline{\tau})}$.}}

\rev{Our strategy is to first establish the functional equation on a dense subset of points. To this end}, we will define a further subset $\Sigma''\left( \mathcal{V},\mathcal{V}^{+}\right) \subset \Sigma\left( \mathcal{V},\mathcal{V}^{+}\right) \subset \operatorname{Spec}\left( \mathbb{T}_{\mathfrak{a}}\right)\times \operatorname{Spec}\left( \mathcal{R}^{\operatorname{univ}}\right)$ consisting of pairs $(f,\theta^t(g))$ where both forms are crystalline. \rev{On this subset, we will show our dual $p$-adic $L$-function $\mathcal{L}'$ interpolates $\Lambda^{[pN]}( f^{*}, g, k-1-t)$ when evaluated at $(f,\theta^t(g))$.} We then proceed to prove a functional equation for $\mathcal{L}$ and $\mathcal{L}'$ on this subset, and extend to the whole weight space by a density argument.


To that end, let us define the subset $\Sigma''\left( \mathcal{V},\mathcal{V}^{+}\right)$ to be $\{\left( f,\theta^{t}g \right) \in \Sigma\left( \mathcal{V},\mathcal{V}^{+}\right): f,g \text{ crystalline} \}$. We will see this is a dense subset of $\operatorname{Spec}\left( \mathbb{T}_{\mathfrak{a}}\right)\times \operatorname{Spec}\left( \mathcal{R}^{\operatorname{univ}}\right)$ in Lemma \ref{lem:densitycrystalline} when we prove the functional equation.

To construct the dual $p$-adic $L$-function $\mathcal{L}'$, we apply the same construction to the "dual family". In light of the above definition of $\Sigma''$, to obtain the dual family on $\Sigma''\left( \mathcal{V,\mathcal{V}^{+}}\right)$, we could just twist $\mathcal{F}^{c}$ and $\mathcal{G}^{[pN]}_{\epsilon^{-1}}$ by their prime-to-$p$ characters (and thus obtain the original universal eigenforms $\mathcal{F}$ and $\mathcal{G}^{[pN]}$). We then define the "dual" $p$-adic family of Eisenstein series $\boldsymbol{E}^{[p]}_{\boldsymbol{k}}\left( q\right)\coloneqq \sum\limits_{p\nmid n}q^{n}\left( \sum\limits_{d\mid n}d^{\boldsymbol{k}-1}\left( \zeta_{N}^{d} + \left( -1\right)^{\boldsymbol{k}}\zeta_{N}^{-d}\right) \right)$, as in Definition 5.3.1 of \cite{LLZ14}.

The dual $p$-adic $L$-function $\mathcal{L}'$ is then defined by $\mathcal{L}'=\lambda_{\mathcal{F}}\left( e\left(\mathcal{G}^{[pN]} \cdot \boldsymbol{E}^{[p]}_{\boldsymbol{k}_{1}-\boldsymbol{k}_{2}} \right)\right)$. By definition, we have $\mathcal{L}'\left( f,\theta^{t}\left( g\right)\right)=\lambda_{f_{\alpha}}\left(\theta^{t}\left( g^{[pN]}\right)\cdot E^{[p]}_{k-l-2t}  \right)$, for all $\left( f,\theta^{t}\left( g\right) \right) \in \Sigma''\left( \mathcal{V},\mathcal{V}^{+}\right)$.

\begin{thm} \label{thm:dualinterpolationformula}
    For all $\left( f, \theta^{t}\left( g\right)\right) \in \Sigma''\left( \mathcal{V},\mathcal{V}^{+}\right)$ (i.e.\ crystalline $f,g$),
   \begin{align*}
     &\mathcal{L}'\left( f,\theta^{t}\left( g\right) \right) \\
    &=i^{k-l-2t}2^{1-k}N^{k-l-2t}   \Lambda^{[pN]}( f^{*}, g, k-1-t) \\
    & \cdot\frac{P_{p}\left( g^{*}, p^{-l-t} \beta\right)}{P_{p}\left(g, \beta^{-1} p^{t}\right)\mathcal{E}\left( f^{*}_{\beta}\right) \mathcal{E}^*\left( f^{*}_{\beta}\right)\langle f^{*}, f^{*} \rangle_{N_{1}}},
\end{align*} where $\beta$ is the non-ordinary root of the Hecke polynomial of $f$ at $p$.
\end{thm}

\begin{proof}
    The proof of Theorem \ref{thm:interpolation formula} applies verbatim here, after applying appropriate twistings.
\end{proof}

\subsection{The interpolation factor at $p$}
\label{modifiedeulerfactoratp}

We briefly verify that the ratio of Euler factors appearing in Theorem \ref{thm:interpolation formula} is consistent with the general conjectures of Panchishkin and Coates--Perrin-Riou summarized in \cite{loe21}.

To the Tate dual $\mathcal{V}^{*}\left( 1\right)$ of $\mathcal{V}$, we define $\left( \mathcal{V}^{*}\left( 1\right)\right)^{+} : = \left( \mathcal{V}/\mathcal{V}^{+}\right)^{*}\left( 1\right)$, i.e.\ $\left( \mathcal{V}^{*}\left( 1\right)\right)^{+}$ is the orthogonal complement of $\mathcal{V}^{+}$ in $\mathcal{V}^{*}\left( 1\right)$. As the notation suggests, $\left( \mathcal{V}^{*}\left( 1\right)\right)^{+}$ to $\mathcal{V}^{*}\left( 1\right)$ is what $\mathcal{V}^{+}$ to $\mathcal{V}$, and defines a Panchishkin subfamily in the sense of \cite{loe21}, though we shall not pursue this explicitly here.

\begin{lemma} \label{localpolynomial}
     Let $\left( V,V^{+}\right)$ be the specialisation of $\left(\mathcal{V},\mathcal{V}^{+} \right)$ at $\left(f, \theta^{t}\left( g\right) \right)$, and let $\left( V^{*}\left( 1\right), V^{*}\left( 1\right)^{+} \right)$ be the specialisation of $ \left( \mathcal{V}^{*}\left( 1\right),\mathcal{V}^{*}\left( 1\right)^{+}\right)$ at $\left( f, \theta^{t}\left( g\right)\right)$. Then
    \begin{enumerate}
        \item $P_{p}\left( g^{*}, p^{-\left( l+t\right)} \alpha \right)=\operatorname{det}\left( 1-\varphi:\boldsymbol{D}_{\operatorname{cris}}\left( V^{+}\right) \right)$.
        \item  $P_{p}\left( g, p^{t}\alpha^{-1} \right)= \operatorname{det}\left( 1-p^{-1}\varphi^{-1}:\boldsymbol{D}_{\operatorname{cris}}\left( V^{+}\right)\right)$
        \item $P_{p}\left( g^{*},p^{-\left(l+t \right)}\beta\right)=\operatorname{det}\left( 1- \varphi: \boldsymbol{D}_{\operatorname{cris}}\left( V/V^{+}\right)\right)$
        \item  $P_{p}\left(g, p^{t}\beta^{-1} \right)=\operatorname{det}\left( 1-\varphi:\boldsymbol{D}_{\operatorname{cris}}\left( \left( V^{*}\left( 1\right)\right)^{+}\right) \right)$
    \end{enumerate}

\end{lemma}

\begin{proof}
    As the proofs are identical, we give a proof to the first equality, and leave the rest to the reader.
    Recall $\psi$ was defined to be the prime-to-$p$ character of $f$. We denote the Galois character associated to $\psi$ by the same symbol, i.e.\ $\psi:G_{\QQ} \rightarrow \mathcal{O}^{\times} $ is the character which sends the geometric Frobenius $\operatorname{Frob}_{\ell}^{-1}$ to $\psi\left( \ell\right)$.
    Let $\alpha\left( f\right) : G_{\QQ_{p}} \rightarrow \mathcal{O}^{\times}$ denote the unramified character satisfying $$\alpha\left( f\right) \left(\operatorname{Frob}^{-1}_{p} \right)= a_{p}\left( f\right), $$
    where as usual $a_{p}\left( f\right)$ is the $p$-th Fourier coefficient of $f$. Since $f$ is ordinary, $\rho_{f} \mid_{G_{\QQ_{p}}}$ has a one-dimensional unramified subrepresentation, and we can describe the shape of $\rho_{f} \mid_{G_{\QQ_{p}}}$ explicitly: $$
    \rho_{f} \mid_{G_{\QQ_{p}}} \thicksim \begin{psmallmatrix}
        \alpha\left( f \right) & * \\
        0 & \alpha\left( f\right)^{-1}  \psi  \varepsilon_{\operatorname{cyc}}^{1-k}    \end{psmallmatrix} .
    $$
    We then observe that by definition,
    $$
    \left( V,V^{+}\right)=\left( V_{f}\otimes V_{g}^{*}\left( 1+t\right), \alpha\left( f\right)\otimes \varepsilon_{\operatorname{cyc}}^{1+t} \otimes V_{g}^{*} \right),
    $$ and
    $$
    \left( V^{*}\left( 1\right), V^{*}\left( 1\right)^{+} \right)=\left( V_{f}^{*}\otimes V_{g}\left( -t\right), \alpha\left( f\right) \otimes \psi^{-1} \otimes \varepsilon_{\operatorname{cyc}}^{k-t-1} \otimes V_{g} \right).
    $$ Let $b_{p}$ be the $p$-th Fourier coefficient of $g$. Then we compute that (building on the work of \cite{Scholl1990}):
         \begin{align*}
        & \operatorname{det}\left(1-\varphi:\boldsymbol{D}_{\operatorname{cris}}\left( \alpha_{f} \otimes \rho_{g}^{*} \otimes \varepsilon_{\operatorname{cyc}}^{1+t} \right) \right) \\
        =& \operatorname{det}\left( 1- \alpha\left( f\right) p^{-1-t} \varphi^{-1}: \boldsymbol{D}_{\operatorname{cris}}\left( \rho_{g}\right) \right) \\
        =& \frac{\left(\alpha\left( f\right)p^{-1-t} \right)^{2}-b_{p}\left(\alpha\left( f\right)p^{-1-t} \right)+\epsilon\left( p\right)p^{l-1}}{\epsilon\left( p\right)p^{l-1}} \\
        =& P_{p}\left( g^{*}, p^{-l-t}\alpha\right).\qedhere
                \end{align*}
\end{proof}

An immediate consequence of the above calculations is that our interpolation formulae (both Theorem \ref{thm:interpolation formula} and Theorem \ref{thm:dualinterpolationformula}) do have the right shape (at least when $f$ and $g$ are both crystalline), as predicted by the second author \cite{loe21}. Take Theorem \ref{thm:interpolation formula} as example, \cite[Conjecture 2.8]{loe21} predicted that our $p$-adic $L$-function should have the following interpolation property: $$
\mathcal{L}\left( x\right)= \left( \text{Euler factor at }p \right)\cdot \dfrac{L\left( M_{x},0\right)}{\left( \operatorname{period}\right)},
$$ where the expected Euler factor is
$$
\operatorname{det}\left( 1-p^{-1}\varphi^{-1}:\boldsymbol{D}_{\operatorname{cris}}\left( V^{+}\right)\right) \cdot \operatorname{det}\left(1-\varphi: \boldsymbol{D}_{\operatorname{cris}}\left( V/V^{+} \right) \right).
$$
But this is precisely what we will see in Equation \eqref{eq:loefflerconjecture} below, for $x \in \Sigma''\left( \mathcal{V},\mathcal{V}^{+}\right)$.

We now use Lemma \ref{localpolynomial} to make a reduction step for the proof of the functional equation for our $p$-adic $L$-functions. By definition, $$P_{p}\left( f,g^{*},l+t\right) \Lambda^{[N]}(f, g^{*}, l+t)=   \Lambda^{[pN]}(f, g^{*}, l+t),$$ where $P_{p}\left( f,g^{*},l+t\right) = \operatorname{det}\left(1-\varphi: \boldsymbol{D}_{\operatorname{cris}}\left(V_{f} \otimes V_{g}^{*}\left( 1+t\right) \right) \right) $ is the reciprocal local Euler factor of the representation $V= \mathcal{V}\left( f, \theta^{t}\left( g\right)\right)= V_{f} \otimes V_{g}^{*}\left( 1+t\right)$ at $p$. Since submodules and quotients of crystalline modules are crystalline, we have \begin{align*}
    \operatorname{det}\left(1-\varphi: \boldsymbol{D}_{\operatorname{cris}}\left( V \right) \right) &= \operatorname{det}\left(1-\varphi: \boldsymbol{D}_{\operatorname{cris}}\left( V^{+} \right) \right) \cdot \operatorname{det}\left(1-\varphi: \boldsymbol{D}_{\operatorname{cris}}\left( V/V^{+} \right) \right).
\end{align*}

By Lemma \ref{localpolynomial}, we find that \begin{align}
    &\frac{P_{p}\left( g, p^{t} \alpha^{-1}\right)}{P_{p}\left(g^{*}, \alpha p^{-l-t}\right)} P_{p}\left( f,g^{*},l+t\right)\\
    =& \operatorname{det}\left( 1-p^{-1}\varphi^{-1}:\boldsymbol{D}_{\operatorname{cris}}\left( V^{+}\right)\right) \cdot \operatorname{det}\left( 1-\varphi:\boldsymbol{D}_{\operatorname{cris}}\left( V^{+}\right) \right)^{-1} \cdot \operatorname{det}\left(1-\varphi: \boldsymbol{D}_{\operatorname{cris}}\left( V \right) \right) \\
    =& \operatorname{det}\left( 1-p^{-1}\varphi^{-1}:\boldsymbol{D}_{\operatorname{cris}}\left( V^{+}\right)\right) \cdot \operatorname{det}\left(1-\varphi: \boldsymbol{D}_{\operatorname{cris}}\left( V/V^{+} \right) \right) \label{eq:loefflerconjecture}.
\end{align}

On the other hand, since by definition we have $V^{*}\left( 1\right)= \mathcal{V}^{*}\left( 1\right)\left( f, \theta^{t}\left( g\right)\right)$, and $P_{p}\left( f^{*},g,k-1-t\right) =\operatorname{det}\left(1-\varphi: \boldsymbol{D}_{\operatorname{cris}}\left( V^{*}\left( 1\right) \right) \right) $, similar to above we have \begin{align*}
    &\frac{P_{p}\left( g^{*}, p^{-l-t} \beta\right)}{P_{p}\left(g, \beta^{-1} p^{t}\right)} \cdot P_{p}\left( f^{*},g, k-1-t\right) \\
    =&\operatorname{det}\left( 1- p^{-1}\varphi^{-1}:\boldsymbol{D}_{\operatorname{cris}}\left( V^{*}\left( 1\right)^{+}\right)\right) \cdot \operatorname{det}\left( 1-\varphi:\boldsymbol{D}_{\operatorname{cris}}\left( \left( V^{*}\left( 1\right)\right)^{+}\right) \right)^{-1} \cdot \operatorname{det}\left(1-\varphi: \boldsymbol{D}_{\operatorname{cris}}\left( V^{*}\left( 1\right) \right) \right)\\
    =& \operatorname{det}\left( 1- p^{-1}\varphi^{-1}:\boldsymbol{D}_{\operatorname{cris}}\left( V^{*}\left( 1\right)^{+}\right)\right) \cdot \operatorname{det}\left( 1-\varphi: \boldsymbol{D}_{\operatorname{cris}}\left( \left( V^{+}\right)^{*}\left( 1\right)\right)  \right)
\end{align*} by noting that $V^{*}\left( 1\right)/V^{*}\left( 1\right)^{+} = V^{*}\left( 1\right)/\left( V/V^{+}\right)^{*}\left( 1\right)=\left( V^{+}\right)^{*}\left( 1\right)$. In particular, by Lemma \ref{localpolynomial}, we can tell that

\begin{prop} \label{modifiedeulerfactor}
    \begin{align*}
        \frac{P_{p}\left( g, p^{t} \alpha^{-1}\right)}{P_{p}\left(g^{*}, \alpha p^{-l-t}\right)}\cdot P_{p}\left( f,g^{*},l+t\right) = \frac{P_{p}\left( g^{*}, p^{-l-t} \beta\right)}{P_{p}\left(g, \beta^{-1} p^{t}\right)} \cdot P_{p}\left( f^{*},g, k-1-t\right).
    \end{align*}
\end{prop}

In particular, to study the complicated comparison \begin{align*}
    \frac{P_{p}\left( g, p^{t} \alpha^{-1}\right)}{P_{p}\left(g^*, \alpha p^{-l-t}\right)}\Lambda^{[pN]}(f, g^{*}, l+t) \thicksim \frac{P_{p}\left( g^{*}, p^{-l-t} \beta\right)}{P_{p}\left(g, \beta^{-1} p^{t}\right)} \Lambda^{[pN]}\left( f^{*},g, k-1-t\right),
\end{align*} it suffices to just compare \begin{align*}
    \Lambda^{[N]}(f, g^{*}, l+t) \thicksim \Lambda^{[N]}(f^{*}, g, k-1-t), 
\end{align*} \rev{since the factor distinguishing $\Lambda^{[pN]}(f, g^{*}, l+t)$ and $\Lambda^{[N]}(f, g^{*}, l+t)$ is precisely the degree $4$ polynomial $P_{p}\left( f,g^{*},l+t\right)$ (and analogously for $ \Lambda^{[pN]}\left( f^{*},g, k-1-t\right)$ and $ \Lambda^{[N]}\left( f^{*},g, k-1-t\right)$). }

\section{The functional equation}

\subsection{The complex functional equation} \label{sec:complexfunctionalequation}

As before we assume both $f$ and $g$ are crystalline.
For each prime $\nu$, define the local complex $\gamma$-factor \[ \gamma_{\nu}\left( s\right)\coloneqq  \gamma_{\nu}\left(f,g,s,\vartheta \right)= \epsilon_{\nu}\left(f,g,s, \vartheta \right)\frac{L_{\nu}\left( f^{*},g^{*},k+l-1-s\right)}{L_{\nu}\left(f,g,s \right)}.\] \rev{Here, ${L_{\nu}\left(f,g,s \right)}$ is the local Rankin-Selberg $L$-factor at $\nu$, defined as the reciprocal of the polynomial $P_{\nu}(f,g,\nu^{-s})$, and $\epsilon_{\nu}(f,g,s,\vartheta)$ is the local epsilon factor with respect to a fixed addtive character $\vartheta: \QQ_{\nu}\to W\left( k\right)^{\times}$.}  In all discussions related to $\gamma_{\nu}\left( f,g,s, \vartheta\right)$ and $\epsilon_{\nu}\left( f,g,s,\vartheta\right) $, $f,g,\vartheta$ will keep fixed, so we omit them in the expression.

Recall the complex functional equation for Rankin-Selberg $L$-functions (from \cite[Proposition 4.1.5]{LLZ14}, see also \cite{li79} for a classical comprehensive account and \cite{Jac72} for a more modern theory):
 \begin{equation*}
     \Lambda\left( f,g,s\right)=\epsilon\left( s\right)\Lambda\left( f^{*},g^{*},k+l-1-s\right)
 \end{equation*} for $\epsilon\left( s\right)$ a global root number. Factoring out local Euler factors at bad primes, we obtain
 \begin{multline} \label{eq:complexfunctionalequation}
     \Lambda\left( f,g,s\right) = \epsilon\left( s\right)\Lambda\left( f^{*},g^{*},k+l-1-s\right) \\
    \Longleftrightarrow \quad \Lambda^{[N]}\left( f,g,s\right) =\prod\limits_{\nu \mid N} \gamma_{\nu}\left(s \right) \Lambda^{[N]}\left( f^{*},g^{*},k+l-1-s\right).
 \end{multline}

Thus to obtain a functional equation of the $p$-adic $L$-function, we need to prove the existence of a universal $\gamma$-factor for each bad prime $\nu \mid N$ that interpolates the complex ones. This will be dealt in Section \ref{sec:universalgammafactor}. For now we prove some technical results which will be crucial to establish the well-definedness of universal $\gamma$-factors.

\subsection{Factorisation of $\mathbb{T}_{\mathfrak{a}} \hat{\otimes} \mathcal{R}^{\operatorname{univ}}$ into three-variables}

Let $\left( \kappa, \Lambda\coloneqq  \mathcal{O}[[\Gamma]] \right)$ denote the universal pair deforming the trivial character from $G_{\QQ,\{p,\infty\}}$ to $\mathbb{F}^{\times}$. We consider deformations of $\overline{\rho}_{2}$ with tame determinant, i.e.\ deformations $\rho$ of $\overline{\rho}_{2}$ such that $\dfrac{\operatorname{det}\left( \rho\right)}{\omega\left( \operatorname{det}\left( \overline{\rho}_{2}\right)\right)}$ is unramified at $p$, where as usual $\omega: \mathbb{F}^{\times} \to \mathcal{O}^{\times}$ is the Teichmuller lift. Let $\boldsymbol{D}^{\operatorname{tame \ det }}_{\mathcal{O}}$ denote the corresponding deformation functor.

\begin{lemma} \label{threevariable}
 The functor $\boldsymbol{D}_{\mathcal{O}}^{\operatorname{tame \ det}}$ is representable. If we let $\left(\rho^{\operatorname{tame \ det}}, \mathcal{R}^{\operatorname{tame \ det}} \right)$ be the universal morphism and its accompanying deformation ring, then
  $$\left( \rho^{\operatorname{univ}}, \mathcal{R}^{\operatorname{univ}}\right) \cong \left( \rho^{\operatorname{tame \ det}},\mathcal{R}^{\operatorname{tame \ det}}\right) \hat{\otimes}_{\mathcal{O}} \left(\kappa, \Lambda\right).$$
\end{lemma}

\begin{proof}
 The first part of this Lemma follows readily because it is a ``determinant" condition.

 For the second part, we construct the explicit isomorphism. For any deformation $\left(\rho,A \right)$, if $\frac{\operatorname{det}\left( \rho \right)}{\omega\left( \operatorname{det}\left( \overline{\rho}_{2}\right)\right)}$ is not unramified at $p$, then there exists a unique character $\tau : \operatorname{Gal}\left( \QQ\left(\zeta_{p^{\infty}} \right)/ \QQ\right) \rightarrow (1 + \mathfrak{m}_A)^\times$ which agrees with $\frac{\operatorname{det}\left( \rho \right)}{\omega\left( \operatorname{det}\left( \overline{\rho}_{2}\right)\right)}$ on the wild inertia group at $p$.

 Since $\tau$ takes values in a pro-$p$ group, it has a unique square root. Define $\rho'\coloneqq  \tau^{-1/2}\otimes \rho$, and so $\rho \cong \rho' \otimes \tau^{1/2}$. By construction, $\rho'$ is also a deformation of $\bar{\rho}_2$, and its determinant is tamely ramified, so it factors through $\rho^{\operatorname{tame \ det}}$. On the other hand, $\tau^{1/2}$ has trivial reduction and is unramified outside $\{p, \infty\}$, and so factors through $\kappa$.

 Conversely, if $\epsilon: G_{\QQ,\{p,\infty \}} \rightarrow A$ is a character that deforms the trivial character, and $\rho$ is a deformation of $\overline{\rho}$ with $\frac{\operatorname{det}\left( \rho \right)}{\omega\left( \operatorname{det}\left( \overline{\rho}_{2}\right)\right)}$ being unramified at $p$, then $\epsilon \otimes \rho$ is again a deformation of $\overline{\rho}_{2}$.
\end{proof}

\begin{lemma}\label{reducedthreevariable}
     The ring $\left(\mathbb{T}_{\mathfrak{a}}\hat{\otimes} \mathcal{R}^{\operatorname{tame \ det}} \right)[1/p]$ is non-zero and reduced.
\end{lemma}

\begin{proof}
 The non-zero part is clear as both factors of $\left(\mathbb{T}_{\mathfrak{a}}\hat{\otimes} \mathcal{R}^{\operatorname{tame \ det}} \right)$ are $\mathcal{O}$-flat, hence $\mathbb{T}_{\mathfrak{a}}\hat{\otimes} \mathcal{R}^{\operatorname{tame \ det}} $ is $p$-torsion-free.

 We note that both $\mathbb{T}_{\mathfrak{a}}$ and $\mathcal{R}^{\operatorname{\operatorname{tame \ det}}}$ are reduced. $\mathbb{T}_{\mathfrak{a}}$ is reduced, because it is a domain; $\mathcal{R}^{\operatorname{\operatorname{tame \ det}}}$ is reduced, because it can be identified with the constant power series in the power series ring $\mathcal{R}^{\operatorname{univ}} \cong \mathcal{R}^{\operatorname{\operatorname{tame \ det}}}[[X]]$ via the isomorphism studied in Lemma \ref{threevariable}, which is reduced by Corollary 3.5 of \cite{boe03}.

 To prove the Lemma, we observe that since $\left(\mathbb{T}_{\mathfrak{a}}\hat{\otimes} \mathcal{R}^{\operatorname{tame \ det}} \right)[1/p]$ is Jacobson, it is enough to prove the completed local ring of $\left(\mathbb{T}_{\mathfrak{a}}\hat{\otimes} \mathcal{R}^{\operatorname{tame \ det}} \right)[1/p]$ at each maximal ideal is reduced. \rev{Let $\left( \cdot \right)_{\eta}$ denotes Berthelot's generic fibre functor. Then we have $\operatorname{MaxSpec}\left( \left(\mathbb{T}_{\mathfrak{a}}\hat{\otimes} \mathcal{R}^{\operatorname{tame \ det}} \right)[1/p]\right)=\left( \operatorname{Spf}\left(\mathbb{T}_{\mathfrak{a}}\hat{\otimes} \mathcal{R}^{\operatorname{tame \ det}} \right)\right)_{\eta}=\left( \operatorname{Spf} \mathbb{T}_{\mathfrak{a}}\right)_{\eta} \times_{L} \left( \operatorname{Spf}\mathcal{R}^{\operatorname{tame\ det }}\right)_{\eta}$ (as sets) by Lemma 7.1.9 of \cite{DeJong1995}}. Let $X\coloneqq  \left( \operatorname{Spf} \mathbb{T}_{\mathfrak{a}}\right)_{\eta} \times_{L} \left( \operatorname{Spf}\mathcal{R}^{\operatorname{tame\ det }}\right)_{\eta}$. The completed local rings of $\left(\mathbb{T}_{\mathfrak{a}}\hat{\otimes} \mathcal{R}^{\operatorname{tame \ det}} \right)[1/p]$ at maximal ideals correspond to the completed local rings of \rev{the structure sheaf $\mathcal{O}_{X}$ on the rigid space $X_{\eta}$} at closed points. We will prove $X$ is reduced, and deduce the reducedness of $\left(\mathbb{T}_{\mathfrak{a}}\hat{\otimes} \mathcal{R}^{\operatorname{tame \ det}} \right)[1/p]$ as a consequence.

 By definition, each factor of $X$, \rev{$(\operatorname{Spf}(\mathbb{T}_{\mathfrak{a}}))_{\eta}$ and $(\operatorname{Spf}(\mathcal{R}^{\operatorname{tame \ det}}))_{\eta}$}, is a union of affinoid subspaces, so it suffices to prove the product of two reduced affinioid spaces (over $L=\operatorname{Frac}\mathcal{O}$) is reduced. Now since $L$ is perfect, the notion of reducedness and geometric reducedness are equivalent, and Ducros proved in \cite[Proposition 6.3 and Remark 6.5]{duc09} that the product of two geometrically reduced spaces is reduced (in the context of Berkovich spaces, which is more general than the one used here), so $X$ is reduced.

 By excellence of affinoid algebras proved in \cite{berk93}, an affinoid algebra is reduced if and only if its associated space is reduced. Thus we see the completed local rings of $\mathcal{O}_{X}$ at closed points, and hence the completed local rings of $\left(\mathbb{T}_{\mathfrak{a}}\hat{\otimes} \mathcal{R}^{\operatorname{tame \ det}} \right)[1/p]$ at maximal ideals, are reduced. The Lemma now follows.
\end{proof}

\begin{lemma} \label{domain}
 Let $\mathfrak{P} \in \operatorname{Spec}\left( \mathbb{T}_{\mathfrak{a}} \hat{\otimes} \mathcal{R}^{\operatorname{tame \ det}}\right)$ be a prime ideal. Then the ring $\left(\mathbb{T}_{\mathfrak{a}} \hat{\otimes}\mathcal{R}^{\operatorname{tame \ det}} / \mathfrak{P}\right) \hat{\otimes} \Lambda$ is an integral domain.
\end{lemma}

\begin{proof}
 Let $A \coloneqq \left(\mathbb{T}_{\mathfrak{a}} \hat{\otimes} \mathcal{R}^{\operatorname{tame \ det}} \right) / \mathfrak{P}$. Then $A$ is a domain, by hypothesis. Moreover, since it is a quotient of a power series ring over $\mathcal{O}$ in finitely many variables, the completed tensor product $A \hat{\otimes} \Lambda$ is simply $A \hat{\otimes} \mathcal{O}[[X]] \cong A[[X]]$, which is a power series ring over a domain, and hence a domain itself.
\end{proof}

 Let $\sigma_{\nu}\coloneqq  \kappa\left( \operatorname{Frob}_{\nu}\right)$, where as usual $\operatorname{Frob}_{\nu}$ denotes an arithmetic Frobenius at $\nu$. Since $\kappa$ is unramified at $\nu \neq p$, $\sigma_{\nu}$ is well-defined.

 \begin{lemma} \label{linearindep}
  Let $A = \left(\mathbb{T}_{\mathfrak{a}} \hat{\otimes} \mathcal{R}^{\operatorname{tame \ det}}\right) / \mathfrak{P}$ be as above; and let $Q\left( X\right) \in A[X, X^{-1}]$ be a non-zero Laurent polynomial with coefficients in $A$. Then $Q\left( \sigma_{\nu}\right) \neq 0$, where we view the evaluation as in $A \hat{\otimes} \Lambda$, via the natural embedding $A[X, X^{-1}] \hookrightarrow A\hat{\otimes} \Lambda$ sending $X$ to $\sigma_\nu$.
 \end{lemma}

\begin{proof}
  By factoring out a suitable negative power of $\sigma_{\nu}$, without loss of generality we may assume $Q\left(X\right)$ is actually a polynomial.

  Since $\nu$ is a prime not dividing $p$, it is not a root of unity in $\zpa$; thus $\langle \nu \rangle = \tfrac{\nu}{\omega(\nu)}$ is an infinite-order element of $1 + p\zpa$. So we may identify $\Lambda$ with a power series ring $\mathcal{O}[[T]]$ in such a way that $\sigma_\nu = (T + 1)^{p^s}$ for some $s \ge 0$. In particular, $\sigma_\nu$ is a monic polynomial in $T$, so $Q(\sigma_\nu)$ is actually a polynomial in $T$ and its leading coefficient is the leading coefficient of $Q \in A[X]$. By a descending induction on the degree, we can recover all of the coefficients of $Q$ from $Q(\sigma_\nu)$.
\end{proof}

\subsection{Universal $\gamma$-factors} \label{sec:universalgammafactor}

 In this Section we show there exists a well-defined universal $\gamma$-factor $\gamma_{\nu}\left(\rho_{\mathcal{A}} \right)$ in the total ring of fraction of $\mathcal{A}$, such that $\gamma_{\nu}\left( \rho_{\mathcal{A}}\right)\left( f,\theta^{t}\left( g\right)\right)=\gamma_{\nu}\left( f,g^{*},l+t, \vartheta\right)=: \gamma_{\nu}\left( l+t\right)$.

 For each prime $\nu \mid N$, let $W_{\nu}=W_{\QQ_{\nu}} \subset G_{\QQ_{\nu}}$ denote the Weil group of the absolute Galois group $G_{\QQ}$ at $\nu$. The half-ordinary Rankin-Selberg universal deformation $\rho_{\mathcal{A}}= \rho^{\ord} \otimes \left( \rho^{\operatorname{univ}}\right)^{*} \left( 1\right)$ restricts to a continuous representation of $W_{\nu}$. Theorem 1.1 of \cite{hm17} ensures the existence of a universal $\gamma$-factor $\gamma_{\nu}\left(\rho_{\mathcal{A}},X \right)$ that interpolates in families (for the explicit description of this $\gamma$-factor, see the Proposition below, or Section 7 of \textit{op. cit.}), i.e.\ for each $\left( f, \theta^{t}\left( g\right)\right) \in \Sigma''\left( \mathcal{V}, \mathcal{V}^{+}\right)$, we have $\gamma_{\nu}\left( \rho_{\mathcal{A}}, 1\right)\left( f, \theta^t \left( g \right) \right)=\gamma_{\nu}\left( l+t\right) $. However, as we observed in the introduction, the ring $\mathcal{T}^{-1}\mathcal{A}[X,X^{-1}]$ may be identically zero when we specialise $X$ to $1$. We shall prove this is not the case by relating $\gamma_{\nu}\left( \rho_{\mathcal{A}},X\right)$ to the universal $\gamma$-factor associated to $\rho^{\ord} \otimes \rho^{\operatorname{tame \ det}}$, as follows.

  The decomposition
  $$
  \left( \rho^{\operatorname{univ}},\mathcal{R}^{\operatorname{univ}}\right) \cong \left( \rho^{\operatorname{tame \ det}},\mathcal{R}^{\operatorname{tame \ det}}\right) \hat{\otimes} \left( \kappa, \Lambda\right)
  $$ of Lemma \ref{threevariable} yields a decomposition of the half-ordinary Rankin-Selberg deformation
  $$
  \rho_{\mathcal{A}}=\rho^{\ord}\otimes \left( \rho^{\operatorname{tame \ det}}\right)^{*}\left( 1\right) \otimes \kappa^{*} =: \widetilde{\rho_{\mathcal{A}}} \otimes \kappa^{*}
  $$

   Define $\widetilde{\mathcal{A}}\coloneqq \mathbb{T}_{\mathfrak{a}}\hat{\otimes}\mathcal{R}^{\operatorname{tame \ det}}$.
  We view $\widetilde{\rho_{\mathcal{A}}}$ as a representation to $\operatorname{GL}_{4}\left( \mathcal{A}\right)$ via the natural inclusion $\widetilde{\mathcal{A}} \hookrightarrow\mathcal{A}$.
\begin{prop}
  $\gamma_{\nu}\left( \rho_{\mathcal{A}},X \right)=\gamma_{\nu}\left( \widetilde{\rho_{\mathcal{A}}},\sigma_{\nu}X\right)$.
\end{prop}

\begin{proof}
    Section 7 of \cite{hm17} gives an explicit construction of the universal $\gamma$-factors involved. We verify this equality using the recipe given there.

    First we note that since $\kappa$ is unramified at $\nu$, if $\rho_{\mathcal{A}}$ is totally wildly ramified (see \textit{loc.cit.} for a definition), then so is $\widetilde{\rho_{\mathcal{A}}}$. In this case, $\gamma_{\nu}\left( \rho_{\mathcal{A}},X\right)\coloneqq \epsilon_{0}\left( \rho_{\mathcal{A}},X\right)$, where $\epsilon_{0}\left( \rho_{\mathcal{A}}\right)$ is the modified $\epsilon$-factor defined in Theorem 5.3 of \cite{yas09}, and $\epsilon_{0}\left( \rho_{\mathcal{A}},X\right)$ is obtained from $\epsilon_{0}\left( \rho_{\mathcal{A}}\right)$ by twisting $\rho_{\mathcal{A}}$ by the unramified character $\operatorname{Frob}^{-1}_{\nu} \mapsto X$. Explicitly, we have $\epsilon_{0}\left( \rho_{\mathcal{A}},X\right)=\epsilon_{0}\left( \rho_{\mathcal{A}}\right)X^{\operatorname{Sw}\left( \rho_{\mathcal{A}}\right)+4}$ (c.f.\ Equation (8) of Theorem 5.1 in \cite{yas09}). On the other hand, since $\rho_{\mathcal{A}}=\widetilde{\rho_{\mathcal{A}}}\otimes \kappa^{*}$, we have $\epsilon_{0}\left( \rho_{\mathcal{A}}\right)=\epsilon_{0}\left( \widetilde{\rho_{\mathcal{A}}}\right)\sigma_{\nu}^{\left( \operatorname{Sw}\left( \widetilde{\rho_{\mathcal{A}}}\right)+4\right)}$, and it follows that $\epsilon_{0}\left( \rho_{\mathcal{A}},X\right)=\epsilon_{0}\left( \widetilde{\rho_{\mathcal{A}}},\sigma_{\nu}X\right)$.

    In the general case, \cite{hm17} defined $\gamma_{\nu}\left( \rho_{\mathcal{A}},X\right)=\gamma_{\nu}\left( \rho_{\mathcal{A}}^{0},X\right)\epsilon_{0}\left( \rho_{\mathcal{A}}^{>0},X\right)$, where $\rho_{\mathcal{A}}=\rho_{\mathcal{A}}^{0}\oplus \rho_{\mathcal{A}}^{>0}$, for $\rho^{0}_{\mathcal{A}}$ tamely ramified and $\rho_{\mathcal{A}}^{>0}$ totally wildly ramified. In light of above, it is enough to assume $\rho_{\mathcal{A}}$ is tamely ramified. In this case, we have $\gamma_{\nu}\left( \rho_{\mathcal{A}},X\right)\coloneqq \epsilon_{0}\left( \rho_{\mathcal{A}}\right)X^{\operatorname{Sw}\left( \rho_{\mathcal{A}}\right)+4}\dfrac{\operatorname{det}_{\rho_{\mathcal{A}}}\left( 1+\sigma+\ldots+\sigma^{\nu-1}\right)\operatorname{Char}_{\rho_{\mathcal{A}}}\left( \operatorname{Frob}^{-1}_{\nu}\right)\left( X\right)}{\operatorname{Char}_{\rho_{\mathcal{A}}}\left(\left( 1+\sigma+\ldots+\sigma^{\nu-1}\right)\operatorname{Frob}^{-1}_{\nu} \right)\left( X\right)}$, where $\sigma$ is a topological generator of the tame inertia. Again using the relation $\rho_{\mathcal{A}}= \kappa^{*} \otimes \widetilde{\rho_{\mathcal{A}}}$, we find that
    $\epsilon_{0}\left( \rho_{\mathcal{A}}\right)=\epsilon_{0}\left( \widetilde{\rho_{\mathcal{A}}}\right)\sigma_{\nu}^{\left( \operatorname{Sw}\left( \widetilde{\rho_{\mathcal{A}}}\right)+4\right)}$, $\operatorname{det}_{\widetilde{\rho_{\mathcal{A}}}}\left( 1+\sigma+\ldots+\sigma^{\nu-1}\right)=\operatorname{det}_{\rho_{\mathcal{A}}}\left( 1+\sigma+\ldots+\sigma^{\nu-1}\right)$ (since $\kappa$ is unramifed and $\sigma$ has trvial image under $\kappa$), and \[\dfrac{\operatorname{Char}_{\widetilde{\rho_{\mathcal{A}}}}\left( \operatorname{Frob}^{-1}_{\nu}\right)\left(\sigma_{\nu} X\right)}{\operatorname{Char}_{\widetilde{\rho_{\mathcal{A}}}}\left(\left( 1+\sigma+\ldots+\sigma^{\nu-1}\right)\operatorname{Frob}^{-1}_{\nu} \right)\left( \sigma_{\nu} X\right)}=\dfrac{\operatorname{Char}_{\rho_{\mathcal{A}}}\left( \operatorname{Frob}^{-1}_{\nu}\right)\left( X\right)}{\operatorname{Char}_{\rho_{\mathcal{A}}}\left(\left( 1+\sigma+\ldots+\sigma^{\nu-1}\right)\operatorname{Frob}^{-1}_{\nu} \right)\left( X\right)}.\] Combining these equalities, we obtain the Proposition.
\end{proof}

We will define the universal $\gamma$-factor of $\rho_{\mathcal{A}}$ as $\gamma_{\nu}\left( \rho_{\mathcal{A}}\right)\coloneqq \gamma_{\nu}\left( \rho_{\mathcal{A}},1\right)$. That is, we take the universal $\gamma$-factor $\gamma_{\nu}\left( \rho_{\mathcal{A}},X\right)$ of \cite{hm17} and set $X$ to $1$. Theorem 1.1 of \cite{hm17} would then ensure that $\gamma_{\nu}\left( \rho_{\mathcal{A}}\right)$ specialise to the usual $\gamma$-factors at all points $\left( f, \theta^t \left( g \right) \right) \in \Sigma''\left( \mathcal{V},\mathcal{V}^{+}\right)$. We have to be cautious here, because by definition, $\gamma_{\nu}\left( \rho_{\mathcal{A}},X\right) \in \mathcal{T}^{-1}\mathcal{A}[X^{-1},X]$, and we need to make sure the set $\mathcal{T}$ contains no zero-divisors upon specialising $X$ to $1$, so that the localised ring is non-zero and the universal $\gamma$-factor is well-defined. In light of the above Proposition, we have $\gamma_{\nu}\left( \rho_{\mathcal{A}},1\right)=\gamma_{\nu}\left( \widetilde{\rho_{\mathcal{A}}},\sigma_{\nu}\right)$, so it is enough to prove $\gamma_{\nu}\left( \widetilde{\rho_{\mathcal{A}}},\sigma_{\nu} \right)$ is well-defined. Let $\widetilde{\mathcal{T}}$ be the set of Laurent polynomials whose first and last coefficients are units in $\widetilde{\mathcal{A}}$, and let $\widetilde{T}$ denote the image of $\widetilde{\mathcal{T}}$ under the map $X \mapsto \sigma_{\nu}$ (viewed as elements in $\mathcal{A}$ via the natural embeddings $\widetilde{\mathcal{A}} \hookrightarrow \mathcal{A}$ and $\Lambda \hookrightarrow \mathcal{A}$). We shall prove $\widetilde{T}$ contains no zero-divisor, which by above analysis will imply $\gamma_{\nu}\left( \rho_{\mathcal{A}}\right)$ is a well-defined element in the non-zero ring $\widetilde{T}^{-1}\mathcal{A}$.

Let $Q\left( X\right) \in \widetilde{\mathcal{T}}$, and write $Q$ for the evaluation of $Q\left( X\right)$ at $X=\sigma_{\nu}$. Then

\begin{thm} \label{welldefinedgammafactor}
    $Q \in \mathcal{A}$ is not a zero-divisor. In particular, the localised ring $\widetilde{T}^{-1}\mathcal{A}$ is non-zero and $\gamma_{\nu}\left( \rho_{\mathcal{A}},1\right)$ is a well-defined element in $ \widetilde{T}^{-1}\mathcal{A}$.
\end{thm}

\begin{proof}
    \rev{To prove that $Q$ is not a zero-divisor, we will show that if $Q \cdot Q'=0 $ in $\mathcal{A}$, then it must be that $Q'=0$}. Let $\left( f,g\right) \in \operatorname{Spec}\left( \mathbb{T}_{\mathfrak{a}} \hat{\otimes} \mathcal{R}^{\operatorname{tame \ det}}[1/p] \right)$. We have $Q\cdot Q' \left( f,g\right)=  Q\left(f,g\right) \cdot Q'\left( f,g\right)=0 \in \left(\mathbb{T}_{\mathfrak{a}} \hat{\otimes} \mathcal{R}^{\operatorname{tame \ det}}\left( f,g\right) \right) \hat{\otimes} \Lambda$. Since $Q\left( f,g\right) \neq 0$ for all $\left( f,g\right)$ by Lemma \ref{linearindep}, and $\left(\mathbb{T}_{\mathfrak{a}} \hat{\otimes} \mathcal{R}^{\operatorname{tame \ det}}\left( f,g\right) \right) \hat{\otimes}\Lambda$ is a domain by Lemma \ref{domain}, this forces $Q'\left(f,g \right)=0$ for all $\left( f,g \right) \in \operatorname{Spec}\left( \mathbb{T}_{\mathfrak{a}}\hat{\otimes} \mathcal{R}^{\operatorname{tame \ det}}[1/p] \right)$.
    Write $Q' = \sum\limits_{i=0}^{\infty} b_{i}X^{\rev{i}}  \in  \mathbb{T}_{\mathfrak{a}}\hat{\otimes} \mathcal{R}^{\operatorname{tame \ det}}[[X]] \cong \mathbb{T}_{\mathfrak{a}}\hat{\otimes} \mathcal{R}^{\operatorname{tame \ det}} \hat{\otimes}\Lambda $. Now $Q'\left( f,g\right)=0$ means for all $i$, $b_{i}\left( f,g\right)=0$. Since this is valid for all $\left( f,g\right) \in \operatorname{Spec}\left( \mathbb{T}_{\mathfrak{a}}\hat{\otimes} \mathcal{R}^{\operatorname{tame \ det}}[1/p] \right)$, we conclude that $b_{i} \in  \bigcap\limits_{\substack{\mathfrak{p} \in \operatorname{Spec}\left(\mathbb{T}_{\mathfrak{a}}\hat{\otimes} \mathcal{R}^{\operatorname{tame \ det}}[1/p] \right)}} \mathfrak{p}$. But by Lemma \ref{reducedthreevariable}, $\mathbb{T}_{\mathfrak{a}}\hat{\otimes} \mathcal{R}^{\operatorname{tame \ det}}[1/p]$ is reduced, and the intersection is zero. Thus $b_{i}=0$ for all $i$, and hence $Q'=0$.
\end{proof}

\begin{cor} \label{universalgammafactor}
     For each $\nu \mid N$, there exists a universal $\gamma$-factor $\gamma_{\nu}\left(\rho_{\mathcal{A}} \right)$ in the total quotient ring of $\mathcal{A}$ with $\gamma_{\nu}\left( \rho_{\mathcal{A}}\right) \left( f, \theta^t \left( g \right) \right)= \gamma_{\nu}\left(f,g^{*}, l+t, \vartheta  \right)$ for all $\left( f, \theta^t \left( g \right) \right) \in \Sigma''\left( \mathcal{V},\mathcal{V}^{+}\right)$, where $\gamma_{\nu}\left(f,g^{*}, l+t, \vartheta  \right)$ is the complex $\gamma$-factor at $\nu$ defined in Section \ref{sec:complexfunctionalequation}.
\end{cor}

\subsection{The functional equation}
\begin{lemma} \label{lem:densitycrystalline}
    The subset $\Sigma''\left( \mathcal{V}, \mathcal{V}^{+}\right) \subset \Sigma\left( \mathcal{V},\mathcal{V}^{+}\right)$ is dense in $\operatorname{Spec}\left( \mathbb{T}_{\mathfrak{a}}\right)\times \operatorname{Spec}\left( \mathcal{R}^{\operatorname{univ}}\right)$.

\end{lemma}

\begin{proof}
    We first prove that crystalline points (of all weights) are dense in $\operatorname{Spec}\left( \mathcal{R}^{\operatorname{univ}}\right)$.
    By Theorem \ref{r=t}, $\mathcal{R}^{\operatorname{univ}}$ is isomorphic to the local Hecke algebra $\mathbb{T}_{\mathfrak{n}}$, which is the completion of $\mathbb{T}_{N_2}$ at the maximal ideal ideal corresponding to $\overline{\rho}_{2}$, so it suffices to prove that crystalline points are dense in $\operatorname{Spec}\left( \mathbb{T}_{\mathfrak{n}}
    \right)$. By Theorem 2.7 of \cite{emer11},\rev{$\mathbb{T}_{N_2}=\prod\limits_{\mathfrak{m}}(\mathbb{T}_{N_2})_{\mathfrak{m}}$} is semi-local, and can be decomposed as a product of finitely many local factors indexed by its maximal ideals, one of which is $\mathbb{T}_{\mathfrak{n}}$. Thus to prove that crystalline points are dense in $\operatorname{Spec}\left( \mathbb{T}_{\mathfrak{n}}\right)$, it suffices to prove they are dense in $\operatorname{Spec}\left( \mathbb{T}_{N_2}\right)$. But this readily follows from Lemma 3 of \cite{gou90} (second part, taking $\nu=0$).

     Now we establish that the set $\Sigma''\left( \mathcal{V}, \mathcal{V}^{+}\right)$ is dense in $\operatorname{Spec}\left( \mathbb{T}_{\mathfrak{a}}\right)\times \operatorname{Spec}\left( \mathcal{R}^{\operatorname{univ}}\right) $. In light of above, it suffices to prove for every crystalline $g$ with weight $l \geq 1$, there exist infinitely many $f \in \operatorname{Spec}\left( \mathbb{T}_{\mathfrak{a}}\right)$ with $\left( f,\theta^{t}\left( g\right)\right) \in \Sigma''\left( \mathcal{V},\mathcal{V^{+}}\right)$.

    By our construction of $\mathbb{T}_{\mathfrak{a}}$, we can find an integer $i \in \mathbb{Z}/\left(p-1 \right)\mathbb{Z}$ such that the the group $\left( \mathbb{Z}/p\mathbb{Z}\right)^{\times} $ acts on $\mathbb{T}_{\mathfrak{a}}$ (via diamond operators) as $\omega^{i}$. In particular, if the weight $k$ of $f$ satisfies $k \equiv i \pmod{ p-1}$, then $f$ has trivial $p$-nebentypus. \rev{Without loss of generality we may assume $k \geq 3$, since we only care about the infinitude of $k$.} It was shown in \cite[p. 183, Section level $N$ versus $Np$]{GOUVEA1992178} that $f$ is automatically crystalline. Thus we see that for each fixed crystalline $g$ of weight $l$, there exist crystalline $f$ of arbitrary high weight $k$ such that $\left( f, \theta^t \left( g \right) \right) \in \Sigma''\left( \mathcal{V},\mathcal{V}^{+}\right)$. The Lemma now follows.
\end{proof}

\begin{thm} \label{thm:p-adicfunctionalequation} We have
    \begin{align*}
        \mathcal{L}= N^{2\left( \boldsymbol{k}_{1}-\boldsymbol{k}_{2} -1 \right)} \gamma\left(\rho_{\mathcal{A}} \right) \mathcal{L}',
    \end{align*}
        where $\gamma\left( \rho_{\mathcal{A}}\right)= \prod\limits_{\nu \mid N}\gamma_{\nu}\left( \rho_{\mathcal{A}}\right)$.
\end{thm}

\begin{proof}
    We first establish that  $$\mathcal{L}\left( f, \theta^t \left( g \right) \right) = N^{2\left( k-l-2t-1\right)} \left(\gamma\left( \rho_{\mathcal{A}}\right) \cdot \mathcal{L}'  \right) \left( f, \theta^t \left( g \right) \right)$$ for all $\left( f, \theta^t \left( g \right) \right) \in \Sigma''\left( \mathcal{V},\mathcal{V}^{+}\right)$.

    By the interpolation formulae of Theorem \ref{thm:interpolation formula} and Theorem \ref{thm:dualinterpolationformula},
     it is immediate that the factor $i^{k-l-2t}2^{1-k}$ will cancel out, and we obtain the power $N^{2\left( k-l-2t-1\right)}$ on the right hand side of the equation.

    By Proposition \ref{modifiedeulerfactor}, the ratios $\dfrac{P_{p}\left( g^{*}, p^{-l-t} \beta\right)}{P_{p}\left(g, \beta^{-1} p^{t}\right)}$ and $\dfrac{P_{p}\left( g, p^{t} \alpha^{-1}\right)}{P_{p}\left(g^{*}, \alpha p^{-l-t}\right)}$ also disappear, and we are left with $$
     \dfrac{\Lambda^{[N]}( f^{*}, g, k-1-t)}{\mathcal{E}\left( f^{*}_{\beta}\right) \mathcal{E}^*\left( f^{*}_{\beta}\right)\langle f^{*}, f^{*} \rangle_{N_{1}}} \thicksim \dfrac{\Lambda^{[N]}(f, g^{*}, l+t)}{\mathcal{E}\left( f_{\alpha}\right) \mathcal{E}^*\left( f_{\alpha}\right)\langle f, f \rangle_{N_{1}}}
    $$ to compare.

    It follows readily from our analysis in Section \ref{sec:complexfunctionalequation}, combined with Corollary \ref{universalgammafactor}, that $$
   \gamma\left(\rho_{\mathcal{A}} \right) \left( f, \theta^t \left( g \right) \right)\Lambda^{[N]}( f^{*}, g, k-1-t)=\Lambda^{[N]}(f, g^{*}, l+t).$$

   On the other hand, the period terms $\mathcal{E}\left( f^{*}_{\beta}\right) \mathcal{E}^*\left( f^{*}_{\beta}\right)\langle f^{*}, f^{*} \rangle_{N_{1}}$ and $\mathcal{E}\left( f_{\alpha}\right) \mathcal{E}^*\left( f_{\alpha}\right)\langle f, f \rangle_{N_{1}}$ \rev{are} also equal. This follows directly from our definitions of $\mathcal{E}\left( \cdot\right)$ and $\mathcal{E}^{*}\left( \cdot\right)$; details can be found in the proof of \cite[Proposition 5.4.4]{LLZ14}. Thus we obtain the desired equality
    $$\mathcal{L}\left( f, \theta^t \left( g \right) \right) = N^{2\left( k-l-2t-1\right)} \left(\gamma\left( \rho_{\mathcal{A}}\right) \cdot \mathcal{L}'  \right) \left( f, \theta^t \left( g \right) \right)$$ for all $\left( f, \theta^t \left( g \right) \right) \in \Sigma''\left( \mathcal{V},\mathcal{V}^{+}\right)$.

    To prove the theorem, we note that by Lemma \ref{lem:densitycrystalline}, the set $\Sigma''\left( \mathcal{V},\mathcal{V}^{+}\right)$ is Zariski dense in $\operatorname{Spec}\left( \mathbb{T}_{\mathfrak{a}}\right) \times \operatorname{Spec}\left( \mathcal{R}^{\operatorname{univ}}\right)$. The functions $\mathcal{L}$ and $\mathcal{L}'$ (defined over $I^{-1}_{\mathfrak{a}} \otimes_{\mathbb{T}_{\mathfrak{a}}} \mathcal{A}$) are meromorphic functions on $\operatorname{Spec}\left( \mathbb{T}_{\mathfrak{a}}\right) \times \operatorname{Spec}\left( \mathcal{R}^{\operatorname{univ}}\right)$, and their poles are disjoint from $\Sigma''\left( \mathcal{V},\mathcal{V}^{+}\right)$. Hence we must have \begin{align*}
        \mathcal{L}= N^{2\left( \boldsymbol{k}_{1}-\boldsymbol{k}_{2} -1 \right)} \gamma\left(\rho_{\mathcal{A}} \right) \mathcal{L}',
    \end{align*} by continuity of $\mathcal{L}$ and $\mathcal{L}'$.
\end{proof}

\section{Proof of Theorem \ref{thm:interpolation formula}}

To unify notations, we let $f_{\beta}$ be the $p$-stablisation of $f$ at the root $\beta$ if $f$ is crystalline, and be $f \otimes\psi_{p}^{-1}$  otherwise. Let $p^{a}$ be the exact level of $f$ at $p$.

Recall $\psi$ and $\psi_{p}$ (resp. $\epsilon$ and $\epsilon_{p}$) are prime-to-$p$ and $p$-part characters of $f$ (resp. $g$), and $g$ is new at level $p^{b}$.

We need to evaluate the linear functional  \begin{align*}
    \lambda_{f_{\alpha}^{c}}\left( \operatorname{Hol}\left(\delta_{l}^{t}\left(g_{\epsilon^{-1}}^{[pN]}  \right) \cdot F^{[p]}_{k-l-2t,\psi_{p}\epsilon_{p}^{-1} } \right) \right) = \left( -1\right)^{t}\lambda_{f_{\alpha}^{c}}\left(
 \operatorname{Hol}\left(
 g_{\epsilon^{-1}}^{[pN]}   \cdot \delta^{t}_{k-l-2t}\left(F^{[p]}_{k-l-2t,\psi_{p}\epsilon_{p}^{-1} } \right)  \right) \right).
\end{align*}
It is given by a ratio of Petersson inner products:

\begin{lemma} \label{linearfunctionalproof}
For all $n \geq \max\{a,1\}$ and $h \in \mathcal{S}_{k}\left( \Gamma_{1}\left( Np^{n}\right) \right)$, we have $$\lambda_{f_{\alpha}^{c}}( h)=
\begin{cases}
    \left( \frac{\psi\left(p\right)}{\alpha p } \right)^{n-1}(p-1)^{-1}   \cdot \dfrac{\langle f_n,h \rangle_{Np^{n}} }{\langle f', f_{\alpha}^{c} \rangle_{N_{1}\left(p\right)} } &\text{if $a = 0$},\\
    \left( \frac{\psi\left(p\right)}{\alpha p} \right)^{n-a} \cdot \dfrac{\langle f_n,h \rangle_{Np^{n}} }{\langle f', f_{\alpha}^{c} \rangle_{N_{1}p^{a}} } &\text{otherwise},
\end{cases}
$$ where $f'= \begin{cases}
     W_{N_{1}p} \left( f_{\beta} \right) &\text{if $a=0$}, \\
     W_{N_{1}p^{a}} \left( f_{\beta} \right) &\text{otherwise},
\end{cases}$ and $f_{n} = \begin{cases}
    f' \mid_k \begin{psmallmatrix}
p^{n-1} &0 \\ 0 & 1
\end{psmallmatrix}  \quad \text{ if a=0} \\
f' \mid_k \begin{psmallmatrix}
p^{n-a} &0 \\ 0 & 1
\end{psmallmatrix} \text{\quad otherwise}.
\end{cases}  $
\end{lemma}

\begin{proof}

    We observe that $f_{n}=f_{\beta}\mid W_{N_{1}p^{n}}$, and $f_{n}$ is an eigenform for the transpose Hecke operators, with eigenvalues being the complex conjugates of those of $f_{\alpha}^{c}$. Consequently, we must have \begin{align*}
        \lambda_{f_{\alpha}^{c}}\left( h\right)= \dfrac{\mid   \left( \mathbb{Z}/p^{a}\right)^{\times}\mid   }{\mid \left(
 \mathbb{Z}/p^{n}\right)^{\times} \mid   } \cdot \dfrac{\langle f_{n},h\rangle_{Np^{n}}}{\langle f_{n},f_{\alpha}^{c}\rangle_{N_{1}p^{a}\left( p^{n-a}\right)}}
    \end{align*} for all $h$ at level $Np^{n}$.

    The result follows using the relation
    \[
        \left\langle f' \mid \begin{psmallmatrix}
            p^{n-a} & 0 \\ 0 & 1
        \end{psmallmatrix},f_{\alpha}^{c}  \right\rangle_{N_{1}p^{a}\left( p^{n-a}\right)}
        =\langle f',f_{\alpha}^{c}\mid U_{p}^{n-a}\rangle_{N_{1} p^{a}}. \qedhere
    \]
\end{proof}

\begin{remark}
    When $f$ is crystalline and $h$ is invariant under diamond actions at $p$, this is \cite[Appendix, Step 1]{loe18}.
\end{remark}

The denominator term is explicitly given by :

\begin{lemma}
\label{linearfunctionaldenominator}
If $f$ is crystalline, then
\[ \langle f', f_{\alpha}^{c} \rangle_{N_{1}(p)} =  \frac{\overline{\lambda_{N_{1}}\left(f \right)}\alpha \mathcal{E}\left(f_{\alpha} \right) \mathcal{E}^{*}\left(f_{\alpha} \right)  } {\psi\left(p\right)} \cdot \langle f, f \rangle_{N_{1}}.\]
Otherwise, we have
\[ \langle f', f_{\alpha}^{c} \rangle_{N_{1}p^{a}}=\begin{cases}
    -\overline{\lambda_{N_{1}}\left( f_{\beta}\right)} \cdot \dfrac{p^{k-2}}{\alpha}\cdot \langle f,f\rangle_{N_{1}p^{a}} &\text{if }a=1, \psi=\operatorname{id}, \\[2mm]
    \overline{\lambda_{N_{1}}\left( f_{\beta}\right)}\left( \dfrac{p^{k-2}}{\alpha}\right)^{a}G\left( \psi_{p}\right)\langle f,f\rangle_{N_{1}p^{a}}  &\text{otherwise. }
\end{cases}\]
\end{lemma}

\begin{proof}
    The crystalline case is from \cite[Appendix, Step 1]{loe18}. We compute the other case:
    \begin{align*}
        \langle  W_{N_{1}p^{a}}\left( f_{\beta}\right), f_{\alpha}^{c} \rangle_{N_{1} p^{a}}&= \overline{\lambda_{N_{1}p^{a}}\left( f_{\beta}\right)}\langle f_{\alpha}^{c}, f_{\alpha}^{c}\rangle_{N_{1} p^{a}} \\
        &= \overline{\lambda_{N_{1}}\left( f_{\beta}\right)} \cdot \overline{\lambda_{p^{a}}\left( f_{\beta}\right)} \cdot \langle \lambda_{N_{1}}\left( f\right)^{-1} \cdot f \mid W_{N_{1}}, \lambda_{N_{1}}\left( f\right)^{-1}
        \cdot f \mid W_{N_{1}} \rangle_{N_{1} p^{a}}\\
        &= \overline{\lambda_{N_{1}}\left( f_{\beta}\right)} \cdot \overline{\lambda_{p^{a}}\left( f_{\beta}\right)} \cdot \dfrac{\lambda_{N_{1}}\left( f\right)}{N_{1}^{k-2}\lambda_{N_{1}}\left( f\right)} \cdot N_{1}^{k-2}\langle f,f\rangle_{N_{1} p^{a}}\\
        &= \overline{\lambda_{N_{1}}\left( f_{\beta}\right)} \cdot \overline{\lambda_{p^{a}}\left( f_{\beta}\right)} \langle f,f\rangle_{N_{1}p^{a}}.
    \end{align*}

We compute that \begin{align*}
    \overline{\lambda_{p^{a}}\left( f_{\beta}\right)} = \dfrac{\left( p^{a}\right)^{k-2}}{\lambda_{p^{a}}\left( f_{\beta}\right)}
    =\dfrac{\left( p^{a}\right)^{k-2}\lambda_{p^{a}}\left( f\right)}{\left( p^{a}\right)^{k-2}\psi\left( p^{a}\right)\psi_{p}\left( -1\right)}
    =\psi\left( p^{-a}\right)\psi_{p}\left( -1\right)\lambda_{p^{a}}\left( f\right).
\end{align*}
If we let $\lambda_{p^{a}}^{\operatorname{AL}}\left( f\right)$ denote the Atkin--Lehner pseudo-eigenvalue defined using Atkin--Li's convention, then $\lambda_{p^{a}}\left( f\right)=\psi_{p}\left( -1\right)\psi\left( p^{a}\right)\lambda_{p^{a}}^{\operatorname{AL}}\left( f\right)$ (c.f.\ \cite[\S 2.5]{klz17}).
By Theorem 2.1 of \cite{al78}, we have \begin{align*}
    \lambda_{p^{a}}^{\operatorname{AL}}\left( f\right)&= \left( \dfrac{p^{k-2}}{\alpha}\right)^{a}G\left( \psi_{p}\right) \text{ if } \operatorname{cond}\left( \psi_{p}\right)=a,  \\
    \lambda_{p^{a}}^{\operatorname{AL}}\left( f\right)&= -\dfrac{p^{k-2}}{\alpha} \text{ if } a=1 \text{ and } \psi_{p}=\operatorname{id}.
\end{align*} Thus
\[ \overline{\lambda_{p^{a}}\left( f_{\beta}\right)}=\begin{cases}
    \left( \dfrac{p^{k-2}}{\alpha}\right)^{a}G\left( \psi_{p}\right) \text{ if } \operatorname{cond}\left( \psi_{p}\right)=a,\\
     -\dfrac{p^{k-2}}{\alpha} \text{ if } a=1 \text{ and } \psi_{p}=\operatorname{id}.
\end{cases} \qedhere \]
\end{proof}

 Let $r \geq a $ be a large enough integer. To ease the notation we write $\chi$ for $\psi_{p}\epsilon_{p}^{-1}$, and write $m$ for $k-l-2t$.

  We now calculate the numerator term. By Lemma \ref{linearfunctionalproof} applied to the function $  g_{\epsilon^{-1}}^{[pN]}   \cdot \delta^{t}_{m}\left(F^{[p]}_{m,\psi_{p}\epsilon_{p}^{-1} } \right) $, the numerator becomes
\begin{equation} \label{eq:51}
\left(-1\right)^{t} \left\langle f_{\beta} \mid_{k}  W_{N_{1}p^{2r}}, g_{\epsilon^{-1}}^{[pN]} \cdot \delta_{m}^{t} \left( F_{m,\chi}^{[p]}\right) \right\rangle_{Np^{2r}}.
\end{equation}

 The first step in this direction is to replace the depleted Eisenstein series $F_{m,\chi}^{[p]}$ by another Eisenstein series, which will be easier for integration. Define \begin{equation*}
     \Tilde{F}=a_0+\sum\limits_{n \geq 1} \left( \sum\limits_{d \mid n, p \nmid d} \chi\left(d \right) d^{m-1}  \left(\zeta_{N}^{n/d}+\left(-1 \right)^{m+\chi}\zeta_{N}^{-n/d} \right) \right) q^{n},
\end{equation*}where the constant term $a_{0}$ will be determined in Lemma \ref{lem:reducetokatoeisenstein}. Note that $\Tilde{F}$ is $p$-ordinary, and its $p$-depletion is precisely $F_{m,\chi}^{[p]}$.

\begin{lemma} \label{linearfunctinalvanish} We have
\[
    \left\langle f_{\beta} \mid_{k}  W_{N_{1}p^{2r}}, g_{\epsilon^{-1}}^{[pN]} \cdot \delta_{m}^{t} \left( F_{m,\chi}^{[p]}\right) \right\rangle_{Np^{2r}}=\left\langle f_{\beta} \mid_{k}  W_{N_{1}p^{2r}}, g_{\epsilon^{-1}}^{[pN]} \cdot \delta_{m}^{t} \left( \Tilde{F}\right) \right\rangle_{Np^{2r}}.
\]
\end{lemma}

\begin{proof}
It suffices to prove
\[
 \lambda_{f_{\alpha}^{c}}\left( g_{\epsilon^{-1}}^{[pN]} \cdot \delta_{m}^{t} \left( F_{m,\chi}^{[p]}- \Tilde{F}\right)\right) = 0.
\]
For this, we note that by construction, there exists a modular form $h$ such that $F_{m,\chi}^{[p]}- \Tilde{F}=h\mid_{m}B_{p}$, where $B_{p}$ is the normalised level-raising operator acting on $q$-expansions as $q \mapsto q^p$. A simple $q$-expansion calculation shows that $g_{\epsilon^{-1}}^{[pN]} \cdot \delta_{m}^{t}\left( h \mid_m B_{p}\right)$ is in the kernel of $U_p$. Since by definition $\lambda_{f_{\alpha}^{c}}\left( \cdot\right)$ factors through the ordinary projector, it must send this form to 0.
\end{proof}

Since $W_{N_{1}p^{2r}}= \langle p^{2r} \rangle_{N_{1}}W_{N_{1}}W_{p^{2r}}$, and the adjoint operator of $W_{p^{2r}}$ is $\begin{psmallmatrix}
p^{2r} & 0 \\  0 & p^{2r} \end{psmallmatrix} W_{p^{2r}}^{-1} = \langle p^{-2r} \rangle_{N_{1}} \cdot \langle -1 \rangle_{p^{2r}} \cdot W_{p^{2r}}$, we have
\begin{multline*}
    \left(-1\right)^{t} \langle f_{\beta} \mid_{k}  W_{N_{1}p^{2r}}, g_{\epsilon^{-1}}^{[pN]} \cdot \delta_{m}^{t} \left( F_{m,\chi}^{[p]}\right) \rangle \\=
    \left(-1\right)^{t} \psi_{p}\left( -1\right) p^{2r} \langle f_{\beta} \mid W_{N_{1}}, \left( g_{\epsilon^{-1}}^{[pN]} \mid W_{p^{2r}} \right) \cdot \left(\delta_{m}^{t} \Tilde{F} \mid W_{p^{2r}}\right) \rangle.
\end{multline*}

 For $(\gamma_{1},\gamma_{2}) \in \left(\QQ/ \mathbb{Z}\right)^{\oplus 2} $, recall Kato's weight $m$ Eisenstein series $F_{\gamma_{1},\gamma_{2}}^{\left( m\right)}$ defined in \cite[\S 3]{kat04}. We omit the superscript $\left(m \right)$ since it is the only weight we will consider in this setting.

Then \begin{lemma} \label{lem:reducetokatoeisenstein}
\begin{align*}
    \Tilde{F} \mid W_{p^{2r}} = p^{r\left( 2m-3\right)} \chi\left( -1\right) \sum\limits_{c \in \left( \mathbb{Z}/p^{r}\right)^{\times}}\chi\left( c\right)F_{0,\frac{c}{p^{r}}+\frac{p^{r}}{N}} \mid_{m} B_{p^{r}}
\end{align*}
\end{lemma}

\begin{proof}
We compute that, via the $q$-expansion formulae given in \cite[Proposition 3.10]{kat04}, we have:
\begin{equation*}
    \Tilde{F}=p^{r\left( m-1\right)} \sum\limits_{c \in \left( \mathbb{Z}/p^{r}\right)^{\times}} \chi\left( c\right) F_{c/p^{r},1/N} \mid_{m}B_{p^{r}}.
\end{equation*}
where we have chosen the undetermined constant term $a_{0}$ of $\Tilde{F}$ to be the constant term of the right-hand-side of the above equation.

Let $x,y,z,w$ be integers chosen as in Section 3.1, so that $W_{p^{2r}}=\begin{psmallmatrix}
p^{2r}x & y \\ p^{2r}Nz & p^{2r}w
\end{psmallmatrix}$. The following identity can be readily verified:
\begin{equation*}
    B_{p^{r}} W_{p^{2r}}=\begin{psmallmatrix}
    p^{r} &0 \\ 0 & p^{r}
    \end{psmallmatrix} \begin{psmallmatrix}
    p^{r}x & y \\ Nz & p^{r}w
    \end{psmallmatrix} B_{p^{r}}.
\end{equation*}

The first matrix acts on weight-$m$ forms by multiplication by $p^{r\left( m-2 \right)}$. The second matrix lies in $\operatorname{SL}_{2}\left( \mathbb{Z}\right)$, and by \cite[\S 3]{kat04}, for all $\gamma \in \operatorname{SL}_{2}\left( \mathbb{Z}\right)$, we have $F_{\left(\gamma_{1},\gamma_{2} \right)} \mid_{m} \gamma = F_{\left(\gamma_{1},\gamma_{2} \right) \cdot \gamma}$. Thus we obtain \begin{align*}
\Tilde{F} \mid W_{p^{2r}}& =  p^{r\left( m-1\right)} \sum\limits_{c \in \left( \mathbb{Z}/p^{r}\right)^{\times}} \chi\left( c\right) F_{c/p^{r},1/N} \mid_{m}B_{p^{r}} W_{p^{2r}}  \\
&= p^{r\left( 2m-3\right)} \chi\left( -1\right) \sum\limits_{c \in \left( \mathbb{Z}/p^{r}\right)^{\times}}\chi\left( c\right)F_{0,\frac{c}{p^{r}}+\frac{p^{r}}{N}} \mid_{m} B_{p^{r}}.\qedhere
\end{align*}
\end{proof}

The numerator term (displayed equation \eqref{eq:51}) becomes:

\begin{align*}
    &\quad \left(-1\right)^{t} \langle f_{\beta} \mid_{k}  W_{N_{1}p^{2r}}, g_{\epsilon^{-1}}^{[pN]} \cdot \delta_{m}^{t} \left( F_{m,\chi}^{[p]}\right) \rangle_{Np^{2r}} \\
    &=\left(-1\right)^{t}p^{r(2t+2m-1)}\epsilon_{p}\left( -1\right)   \langle f_{\beta} \mid W_{N_{1}}, \left( g_{\epsilon^{-1}}^{[pN]} \mid W_{p^{2r}} \right)   \cdot \left( \sum\limits_{c \in \left(\mathbb{Z}/p^{r}\mathbb{Z}\right)^{\times}} \chi(c) \delta_{m}^{t}  \cdot  \left( F_{0, c/p^{r}+p^{r}/N} \mid B_{p^{r}}\right) \right) \rangle_{Np^{2r}} \\
    &= \left(-1\right)^{t}p^{r(2t+2m)}(1-\dfrac{1}{p})  \epsilon_{p}\left( -1\right) \langle f_{\beta} \mid W_{N_{1}}, \left( g_{\epsilon^{-1}}^{[pN]} \mid W_{p^{2r}} \right) \cdot \delta_{m}^{t}\left(   F_{0,1/p^{r}+p^{r}/N} \mid B_{p^{r}} \right) \rangle_{Np^{2r}}.
\end{align*}

\begin{lemma} \label{lem:reducelevel}
\begin{align*}
    &\left(-1\right)^{t} \langle f_{\beta} \mid_{k}  W_{N_{1}p^{2r}}, g_{\epsilon^{-1}}^{[pN]} \cdot \delta_{m}^{t} \left( F_{m,\chi}^{[p]}\right) \rangle_{Np^{2r}}\\ =& \left( -1\right)^{t}\epsilon_{p}\left( -N\right) \psi^{-1}_{p}\left( N\right)\psi\left( p\right)^{-2r}\epsilon\left( p\right)^{2r}  p^{2r\left( t+m\right)} \left( 1-\dfrac{1}{p}\right)  \cdot \langle f_{\beta} \mid W_{N_{1}}, \left( g_{\epsilon^{-1}}^{[pN]} \mid W_{p^{2r}} \right) \cdot \delta_{m}^{t}\left(  F_{0,1 / N p^{2 r}}  \right) \rangle_{Np^{2r}}.
\end{align*}
\end{lemma}

\begin{proof}
We observe $F_{0,1/p^{r}+p^{r}/N}$ is the image of $F_{0,1/Np^{r}}$ under the action of $\langle N \rangle_{p} \langle p^{2r} \rangle_{N}$. We have \begin{align*}
      &\langle f_{\beta} \mid W_{N_{1}}, \left( g_{\epsilon^{-1}}^{[pN]} \mid W_{p^{2r}} \right) \cdot \delta_{m}^{t}\left(   F_{0,1/p^{r}+p^{r}/N} \mid B_{p^{r}} \right) \rangle_{Np^{2r}} \\
      =& \langle f_{\beta} \mid W_{N_{1}}, \left( g_{\epsilon^{-1}}^{[pN]} \mid W_{p^{2r}} \right) \cdot \delta_{m}^{t}\left(   F_{0,1/Np^{r}} \mid\langle N \rangle_{p} \langle p^{2r} \rangle_{N}  B_{p^{r}} \right) \rangle_{Np^{2r}} \\
      =&\langle f_{\beta} \mid W_{N_{1}} \mid \langle N^{-1} \rangle_{p} \langle p^{-2r} \rangle_{N} ,  \left( g_{\epsilon^{-1}}^{[pN]} \mid W_{p^{2r}} \mid \langle N^{-1} \rangle_{p} \langle p^{-2r} \rangle_{N} \right) \cdot \delta_{m}^{t}\left(   F_{0,1/Np^{r}} \mid B_{p^{r}} \right) \rangle_{Np^{2r}}\\
      =&\psi_{p}^{-1}\left( N\right) \psi^{-1}\left( p\right)^{2r}\epsilon\left( p \right)^{2r} \epsilon_{p}\left( N\right)\langle f_{\beta} \mid W_{N_{1}}, \left( g_{\epsilon^{-1}}^{[pN]} \mid W_{p^{2r}} \right) \cdot \delta_{m}^{t}\left(   F_{0,1/Np^{r}} \mid B_{p^{r}} \right) \rangle_{Np^{2r}}.
\end{align*}

Again by $q$-expansion calculations, one finds that  \begin{equation*}
    F_{0,1 / N p^{r}} \mid B_{p^{r}}=p^{-r} \sum_{c \in \mathbb{Z} / p^{r}\mathbb{Z}} F_{0,\left(1+N p^{r} c\right) / N p^{2 r}},
\end{equation*} and so :\begin{align*}
    &\langle f_{\beta} \mid W_{N_{1}}, \left( g_{\epsilon^{-1}}^{[pN]} \mid W_{p^{2r}} \right) \cdot \delta_{m}^{t}\left(   F_{0,1/Np^{r}} \mid B_{p^{r}} \right) \rangle_{Np^{2r}} \\
    =& \langle f_{\beta} \mid W_{N_{1}}, \left( g_{\epsilon^{-1}}^{[pN]} \mid W_{p^{2r}} \right) \cdot \delta_{m}^{t}\left(   \sum_{c \in \mathbb{Z} / p^{r}\mathbb{Z}} F_{0,\left(1+N p^{r} c\right) / N p^{2 r}} \right) \rangle_{Np^{r}\left(p^{r}\right)} \\
    =& \langle f_{\beta} \mid W_{N_{1}}, \left( g_{\epsilon^{-1}}^{[pN]} \mid W_{p^{2r}} \right) \cdot \delta_{m}^{t}\left(  F_{0,1 / N p^{2 r}}  \right) \rangle_{Np^{2r}}.
\end{align*}

Putting everything together, we obtain the Lemma.
\end{proof}

\begin{lemma} \label{llz423}
Let $S(pN)$ denote the set of integers whose prime factors all divide $pN$. Then
    \begin{align*}
        &\langle f_{\beta} \mid W_{N_{1}}, \left( g_{\epsilon^{-1}}^{[pN]} \mid W_{p^{2r}} \right) \cdot \delta_{m}^{t}\left( F_{0,1/Np^{2r}} \right) \rangle_{Np^{2r}} \\
        =& 2^{1-k} i^{k-l} \left( Np^{2r} \right)^{2+2t-k+l} \Lambda^{[pN]}( f, g^{*}, l+t) \cdot C\left(\left( f_{\beta} \mid W_{N_{1}}\right)^{*}, g_{\epsilon^{-1}}^{[pN]} \mid W_{p^{2r}}, l+t\right),
    \end{align*} where \begin{align*}
    & C\left(\left( f_{\beta} \mid W_{N_{1}}\right)^{*}, g_{\epsilon^{-1}}^{[pN]} \mid W_{p^{2r}}, l+t\right)
    \coloneqq    \sum\limits_{n \in S(pN)} a_n(\left( f_{\beta} \mid W_{N_{1}}\right)^{*}) a_n(g_{\epsilon^{-1}}^{[pN]} \mid W_{p^{2r}}) n^{-l-t}.
\end{align*}
\end{lemma}

\begin{proof}
    This is Theorem 4.2.3 of \cite{LLZ14}.
\end{proof}

We need to evaluate the quantity $\sum\limits_{n \in S(pN)} a_n(\left( f_{\beta} \mid W_{N_{1}}\right)^{*}) a_n(g_{\epsilon^{-1}}^{[pN]} \mid W_{p^{2r}}) n^{-l-t}$. To ease the notation we write $s$ for $-l-t$. The key input is Lemma \ref{keylemma}.

\begin{lemma} \label{keylemma}
Assume $h$ is a newform of weight $l$, level $p^{b}M$ and character $\epsilon_{M}=\epsilon_{M} \cdot \epsilon_{p}$, where $\left( p,M\right)=1$. Let $R=p^{c+b}M$ for $c \geq 0$, which we assume to be large enough so that $h^{[p]}$ has level $R$. Then we have
     $$\sum\limits_{u \geq 0} X^{u} a_{p^{u}}\left(\left.h_{\epsilon_{M}^{-1}}^{[p]}\right|_{l} W_{p^{b+c}}\right)=\epsilon^{-1}_{M}(p^{b+c})\lambda_{p^b}\left(h\right)\left(p^{l-1} X\right)^{c} \frac{P_{p}\left( h, p^{-l} X^{-1}\right)}{P_{p}\left(h^{*}, X\right)}.$$

\end{lemma}

\begin{proof}

     The key observation is that \begin{align*} \label{eq:keylemma1}
         \sum\limits_{u \geq 0} X^{u} a_{p^{u}}\left(\left.h^{[p]}\right|_{l} W_{R}\right)=\lambda_{p^{b}M}(h)\left(p^{l-1} X\right)^{c} \frac{P_{p}\left(h, p^{-l} X^{-1}\right)}{P_{p}\left(h^{*}, X\right)}.
     \end{align*}
     This follows readily from the functional equations of $h$ and $h^{[p]}$, combined with the relation
         $\Lambda\left( h^{[p]},l-s\right)=\Lambda\left( h,l-s\right) P_{p}\left( h,p^{-\left( l-s\right)}\right).$

 The Lemma now follows using the above displayed equation, the relation $W_{R}=\langle p^{b+c}\rangle_{M}W_{M}W_{p^{b+c}}$, and the fact that $\lambda_{p^{b}M}\left( h\right)=\lambda_{p^{b}}\left( h\right)\lambda_{M}\left( h\right)$ under the convention of \cite{klz17}.
\end{proof}

\begin{lemma} \label{lem:correctionterm}  Let $b$ be the power of $p$ at which $g$ is new. Then
\begin{align*}
    &\sum\limits_{n \in S(pN)} a_n(\left( f_{\beta} \mid W_{N_{1}}\right)^{*}) a_n(g_{\epsilon^{-1}}^{[pN]} \mid W_{p^{2r}}) n^s \\
    =& \overline{\lambda_{N_{1}}\left( f_{\beta}\right)} \epsilon^{-1}\left(p^{2r} \right)  \lambda_{p^b}\left(g \right) \left(\alpha p^{s+l-1} \right)^{2r-b} \frac{P_{p}\left( g, p^{-l-s} \alpha^{-1}\right)}{P_{p}\left(g^{*}, \alpha p^{s}\right)}.
\end{align*}

\end{lemma}

\begin{proof}

We compute that
\begin{align*}
     &\sum\limits_{n \in S(pN)} a_n(\left( f_{\beta} \mid W_{N_{1}}\right)^{*}) a_n(g_{\epsilon^{-1}}^{[pN]} \mid W_{p^{2r}}) n^s \\
     =&\sum\limits_{n \in S(pN)} a_n(\left( f_{\beta} \mid W_{N_{1}}\right)^{*}) a_n\left(\left(g_{\epsilon^{-1}}^{[p]} \mid W_{p^{2r}}\right)^{[N]} \right) n^s \\
     =& \overline{\lambda_{N_{1}}\left( f_{\beta}\right)} \cdot \sum\limits_{u \geq 0} \alpha^{u} a_{p^{u}}\left(g_{\epsilon^{-1}}^{[p]} \mid W_{p^{2r}} \right) p^{us},
\end{align*}  where we have written $\lambda_{N_{1}}\left( f_{\beta}\right)$ for $\lambda_{N_{1}}\left( f\right)$ in the crystalline case for the ease of notations. Then Lemma \ref{keylemma} yields
\begin{align*}
    &\sum\limits_{u \geq 0} \overline{a_{p^{u}} \left( f_{\beta} \mid W_{N_{1}} \right) } a_{p^{u}}\left(g_{\epsilon^{-1}}^{[p]} \mid W_{p^{2r}} \right) p^{us} \\
    =& \sum\limits_{u \geq 0} \overline{\lambda_{N_{1}}\left( f_{\beta}\right)} \left(  \alpha p^s \right)^{u} a_{p^{u}}\left(g_{\epsilon^{-1}}^{[p]} \mid W_{p^{2r}} \right) \\
    =&\ \overline{\lambda_{N_{1}}\left( f_{\beta}\right)} \epsilon^{-1}\left( p^{2r}\right) \lambda_{p^{b}}\left( g\right)\left(\alpha p^{s+l-1}\right)^{2r-b} \dfrac{P_{p}\left( g,p^{-l-s} \alpha^{-1}\right)}{P_{p}\left( g^{*}, \alpha p^{s}\right)}.\qedhere
\end{align*}
\end{proof}

Putting everything together, we obtain the theorem.

%

\providecommand{\bysame}{\leavevmode\hbox to3em{\hrulefill}\thinspace}
\providecommand{\MR}[1]{}
\renewcommand{\MR}[1]{%
 MR \href{https://www.ams.org/mathscinet-getitem?mr=#1}{#1}.
}
\providecommand{\href}[2]{#2}
\newcommand{\articlehref}[2]{\href{#1}{#2}}

\end{document}